\documentclass[12pt,english,smallextended]{article}
\usepackage[T1]{fontenc}
\usepackage[latin9]{inputenc}
\usepackage[a4paper]{geometry}
\usepackage{color}
\usepackage{float}
\usepackage{mathrsfs}
\usepackage{algorithm2e}
\usepackage{amsmath}
\usepackage{amsthm}
\usepackage{amssymb}
\usepackage{graphicx}
\usepackage{setspace}

\makeatletter

\usepackage{xfrac}

\newcommand{\lyxaddress}[1]{
	\par {\raggedright #1
	\vspace{1.4em}
	\noindent\par}
}
\theoremstyle{plain}
\newtheorem{thm}{\protect\theoremname}
\theoremstyle{definition}
\newtheorem{defn}[thm]{\protect\definitionname}
\theoremstyle{plain}
\newtheorem{lem}[thm]{\protect\lemmaname}
\theoremstyle{remark}
\newtheorem{rem}[thm]{\protect\remarkname}
\newenvironment{lyxlist}[1]
	{\begin{list}{}
		{\settowidth{\labelwidth}{#1}
		 \setlength{\leftmargin}{\labelwidth}
		 \addtolength{\leftmargin}{\labelsep}
		 }}
	{\end{list}}
\theoremstyle{plain}
\newtheorem{fact}[thm]{\protect\factname}
\theoremstyle{definition}
\newtheorem{example}[thm]{\protect\examplename}

\date{}
\usepackage{mathtools}
\usepackage{graphicx}
\usepackage[]{xcolor}
\usepackage{cite}
\usepackage{amsmath, amssymb}
\usepackage[margin=1cm,%
            font=small,%
            format=hang,%
            labelsep=period,%
            labelfont=bf]{caption}

\definecolor{blue3}{RGB}{0,0,255}
\definecolor{skyblue3}{RGB}{0,102,255}
\definecolor{green4}{RGB}{0,153,0}
\definecolor{green5}{RGB}{0,102,0}
\definecolor{red5}{RGB}{153,0,0}

\newcommand{\nlessdot}{\mathrel{\lessdot\!\!\!\!\!\!\not\;\;}}
\newcommand{\ah}{\raisebox{-.29\height}{\includegraphics[scale=0.15]{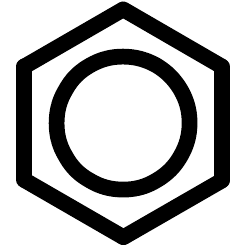}}}

\newcommand{\kh}{\raisebox{-.29\height}{\includegraphics[scale=0.15]{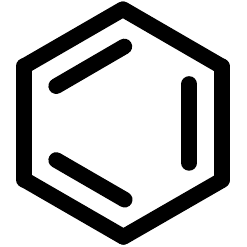}}}
\newcommand{\soc}{\textsc{dib}}
\newcommand{\Soc}{\textup{D}\textsc{ib}}

\SetKwBlock{Either}{either}{}
\SetKwBlock{Or}{or}{end}

\SetCommentSty{mycommfont}

\makeatother

\usepackage{babel}
\providecommand{\definitionname}{Definition}
\providecommand{\examplename}{Example}
\providecommand{\factname}{Fact}
\providecommand{\lemmaname}{Lemma}
\providecommand{\remarkname}{Remark}
\providecommand{\theoremname}{Theorem}

\begin{document}
\begin{doublespace}

\title{\noindent \textbf{\vspace*{0.6cm}
}\\
\textbf{\huge{}ZZ}\textbf{ Polynomials}\\
\vspace*{0.3cm}
\textbf{of Regular }$\boldsymbol{m}$\textbf{-tier Benzenoid Strips}\\
\vspace*{0.3cm}
\textbf{as Extended Strict Order Polynomials}\\
\vspace*{0.3cm}
\textbf{of Associated Posets}\\
\vspace*{0.3cm}
\textbf{Part 1. Proof of Equivalence}}
\end{doublespace}

\author{\noindent \textbf{Johanna Langner}\textit{$^{1}$}\textbf{ and Henryk
A. Witek}\textit{$^{1,2}$}}
\maketitle

\lyxaddress{\noindent \vspace{-2cm}
}

\lyxaddress{\begin{center}
\textit{$^{1}$Department of Applied Chemistry and Institute of Molecular
Science,}\linebreak{}
 \textit{National Yang Ming Chiao Tung University, Hsinchu, Taiwan}\linebreak{}
\textit{$^{2}$Center for Emergent Functional Matter Science, National
Yang Ming Chiao Tung University, Hsinchu, Taiwan}\\
\vspace{0.2cm}
\texttt{\textbf{\textit{\emph{e-mail: johanna.langner@arcor.de, hwitek@mail.nctu.edu.tw}}}}\vspace{-1cm}
\par\end{center}}

\begin{center}
(Received March 08, 2021)\vspace*{0.4cm}
\par\end{center}
\begin{abstract}
\noindent In Part 1 of the current series of papers, we demonstrate
the equivalence between the Zhang-Zhang polynomial $\text{ZZ}(\boldsymbol{S},x)$
of a Kekuléan regular $m$-tier strip $\boldsymbol{S}$ of length
$n$ and the extended strict order polynomial $\text{E}_{\mathcal{S}}^{\circ}(n,x+1)$
of a certain partially ordered set (poset) $\mathcal{S}$ associated
with $\boldsymbol{S}$. The discovered equivalence is a consequence
of the one-to-one correspondence between the set $\left\{ K\right\} $
of Kekulé structures of $\boldsymbol{S}$ and the set $\left\{ \mu:\mathcal{S}\supset\mathcal{A}\rightarrow\left[\,n\,\right]\right\} $
of strictly order-preserving maps from the induced subposets of $\mathcal{S}$
to the interval $\left[\thinspace n\thinspace\right]$. As a result,
the problems of determining the Zhang-Zhang polynomial of $\boldsymbol{S}$
and of generating the complete set of Clar covers of $\boldsymbol{S}$
reduce to the problem of constructing the set $\mathcal{L}(\mathcal{S})$
of linear extensions of the corresponding poset $\mathcal{S}$ and
studying their basic properties. In particular, the Zhang-Zhang polynomial
of $\boldsymbol{S}$ can be written in a compact form as
\[
\text{ZZ}(\boldsymbol{S},x)=\sum_{k=0}^{\left|\mathcal{S}\right|}\sum_{w\in\mathcal{L}(\mathcal{S})}\binom{\left|\mathcal{S}\right|-\text{fix}_{\mathcal{S}}(w)}{\,\,k\,\,\hspace{1pt}-\text{fix}_{\mathcal{S}}(w)}\binom{n+\text{des}(w)}{k}\left(1+x\right)^{k},
\]
where $\text{des}(w)$ and $\text{fix}_{\mathcal{S}}(w)$ denote the
number of descents and the number of fixed labels, respectively, in
the linear extension $w\in\mathcal{L}(\mathcal{S})$. A practical
guide and a four-step, completely automatable algorithm for computing
$\text{E}_{\mathcal{S}}^{\circ}(n,x+1)$ of an arbitrary strip $\boldsymbol{S}$,
followed by a complete account of ZZ polynomials for all regular $m$-tier
benzenoid strips $\boldsymbol{S}$ with $m=1\text{--}6$ and arbitrary
$n$ computed using the discovered equivalence between $\text{ZZ}(\boldsymbol{S},x)$
and $\text{E}_{\mathcal{S}}^{\circ}(n,x+1)$, are presented in Parts
2 and 3, respectively, of the current series of papers {[}J. Langner,
H.~A. Witek, MATCH Commun. Math. Comput. Chem. (2021) (submitted),
\emph{ibid. }(submitted){]}. 

\noindent We would like to stress that the pursued by us approach
is unprecedented in the existing literature on chemical graph theory
and therefore it seems to deserve particular attention of the community,
despite of its quite difficult exposition and connection to advanced
concepts in order theory.
\end{abstract}
\maketitle \thispagestyle{empty} \baselineskip=0.30in

\section{General introduction}

The theory of Clar covers of benzenoids dates back to the early seminal
work of Clar \cite{clar1972thearomatic}, who suggested that the most
chemically stable resonance structures of benzenoids are those with
the maximal number of aromatic sextets. This maximal number of aromatic
sextets that can be accommodated in a given benzenoid $\boldsymbol{B}$\textemdash referred
to in the modern literature as the \emph{Clar number} of $\boldsymbol{B}$
and denoted as $Cl$\textemdash constitutes an important topological
invariant of $\boldsymbol{B}$. A considerable effort has been invested
in the determination of $Cl$ for various classes of benzenoids or
generalized benzenoids \cite{aihara2014constrained,bavsic2017onthe,berlic2015equivalence,chan2015alineartime,chen2010zhangzhang,cruz2012convexhexagonal,klavvzar2002clarnumber,langner2018tilings4,pletervsek2016equivalence,salem2009theclar,tratnik2016resonance,vesel2014fastcomputation,ahmadi_computing_2016,abeledo_unimodularity_2007,ashrafi_relations_2009,ashrafi_clar_2010,balaban_using_2011,berczi-kovacs_complexity_2018,carr_packing_2014,chapman_pairwise_2018,hartung_clar_2014,salem_clar_2004,zhou_relations_2008,zhou_clar_2015}.
In general, a \emph{Clar structure} realizing the maximal number $Cl$
of aromatic sextets is not unique; the number of Clar structures is
denoted by $c_{Cl}$ and constitutes yet another important topological
invariant of $\boldsymbol{B}$. A related, well-established and thoroughly-studied
concept is the Kekulé count $K\left\{ \boldsymbol{B}\right\} $ denoting
how many resonance structures of $\boldsymbol{B}$ can be constructed
using only double bonds and no aromatic sextets \cite{kekule1866untersuchungen,cyvin1988kekulestructures}.
These two numbers, $c_{0}\equiv K\left\{ \boldsymbol{B}\right\} $
and $c_{Cl}$, can be considered as the beginning and the end-point
of a sequence $c_{0},c_{1},\ldots,c_{Cl}$ denoting the cardinalities
of the sets of Clar covers of different order, with $c_{k}$ corresponding
to the number of generalized resonance structures of $\boldsymbol{B}$
constructed using exactly $k$ aromatic sextets. The generating function
$\text{ZZ}(\boldsymbol{B},x)$ for this sequence was introduced to
chemical graph theory by Zhang and Zhang as the \emph{Clar covering
polynomial}, but in the modern literature it is more often referred
to as the \emph{Zhang-Zhang polynomial} or the \emph{ZZ polynomial}.
Zhang and Zhang showed \cite{zhang1996theclar,zhang1996theclar2,zhang1997theclar,zhang2000theclar,zhang_advances_2011}
that $\text{ZZ}(\boldsymbol{B},x)$ has a number of inviting recurrence
properties, which make its determination much easier than finding
any single of the topological invariants of $\boldsymbol{B}$. These
results stimulated Gutman, Furtula, and Balaban \cite{gutman2006algorithm}
and later Chou and Witek \cite{chou2012analgorithm} to design an
algorithm capable of fast and robust computation of $\text{ZZ}(\boldsymbol{B},x)$
using the concept of recursive decomposition. The resulting computer
program (ZZCalculator) \cite{chou2012analgorithm,chou2012zhangzhang}
for determination of Zhang-Zhang polynomials of arbitrary benzenoid
structures has later been augmented with a graphical interface (ZZDecomposer)
\cite{chou2014zzdecomposer,ZZDecomposerDownload1,ZZDecomposerDownload2}
allowing one for generation of benzenoid graphs, computation of $\text{ZZ}(\boldsymbol{B},x)$,
and analysis of the recursive decomposition pathways. ZZDecomposer
has been used in many applications \cite{chou2014determination,chou2014closedtextendashform,chou2015twoexamples,chou2016closedform,He2020determinantal,He2020hexacorners,He2020JohnSachs,He2020parachains,He2020ribbons,langner2017zigzagConnectivityGraph2,langner2018multiplezigzag1,page2013quantum,witek2015zhangzhang,witek2017zhangzhang,witek2020flakes1-8,witek2020fullerenes,witek2020overlappingparas}
to discover and formally prove closed-form formulas of ZZ polynomials
for various classes of elementary and composite benzenoids. At present,
the most important unsolved problems in the theory of ZZ polynomials
are the determination of $\text{ZZ}(\boldsymbol{B},x)$ for oblate
rectangles $Or\left(m,n\right)$ and hexagonal graphene flakes $O\left(k,m,n\right)$.

Completely new vistas in the Clar theory have been recently opened
by the development of the interface theory of benzenoids \cite{langner2019IFTTheorems5,langner2019BasicApplications6,langner2018algorithm3,witek2020overlappingparas,He2020ribbons}.
It has been demonstrated that the description of resonance structures
of a benzenoid $\boldsymbol{B}$ can be reduced to studying the covering
characters of its interfaces. The number of covered edges in each
interface and the relative distribution of the covered edges between
the consecutive interfaces of $\boldsymbol{B}$ is regulated by the
basic tenets of the interface theory (Theorems\,11, 16, and 21 of
\cite{langner2019IFTTheorems5}), allowing one to express uniquely
each Clar cover of $\boldsymbol{B}$ as a sequence of covered interface
bonds of $\boldsymbol{B}$. The generation of the full set of Clar
covers can then be conveniently performed by considering all possible
distributions of covered interface bonds in $\boldsymbol{B}$ that
satisfy the interface theory requirements. In the current work, we
communicate a very important connection discovered by us recently
in this context. Namely, we show that the distribution of double interface
bonds in Kekulé structures of regular benzenoid strips can be very
naturally expressed using the formalism of partially ordered sets
(posets). The existing language and the machinery of the poset theory
allows us to articulate many concepts of the interface theory in a
particularly natural and compact form. Before jumping into technicalities,
we find it appropriate to outline our main findings here. For definitions
of basic terms in the poset theory see Section~\ref{sec:Glossary}
or Stanley's textbook \cite{stanley_enumerative_1986}.

\section{Outline of the results}

The main results obtained in the current work can be briefly summarized
as follows.
\begin{itemize}
\item Every Kekuléan regular $m$-tier strip $\boldsymbol{S}$ of length
$n$ can be uniquely associated with a certain poset $\mathcal{S}$.
\item Every Clar cover of $\boldsymbol{S}$ can be associated with a unique
linear extension of an induced subposet $Q\subset\mathcal{S}$. The
number of Clar covers associated with a linear extension $v$ of $Q$
is given by $2^{\left|Q\right|}\thinspace\binom{n+\text{des}(v)}{\left|Q\right|}$,
where $\text{des}(v)$ denotes the number of descents in $v$. Each
of these $2^{\left|Q\right|}\thinspace\binom{n+\text{des}(v)}{\left|Q\right|}$
Clar covers differs from each other by various distributions of covering
characters (proper sextet \kh~or aromatic sextet \ah) among available
positions ($\left|Q\right|$ distinct entries selected from the sequence
$1,2,\ldots,n+\text{des}(v)$). 
\item The set $\left\{ K\right\} $ of Kekulé structures of $\boldsymbol{S}$
stands in a one-to-one relationship to the set $\left\{ \mu:\mathcal{S}\supset\mathcal{A}\rightarrow\left[\,n\,\right]\right\} $
of strictly order-preserving maps from induced subposets of $\mathcal{S}$
to the interval $\left[\thinspace n\thinspace\right]$. This correspondence
is established by a complementary pair of difficult and technical
Lemmata~\ref{thm:Deconstruct_CC} and~\ref{thm:Construct_CC}.
\item The ZZ polynomial $\text{ZZ}(\boldsymbol{S},x)$ of $\boldsymbol{S}$
is identical to the extended strict order polynomial $\text{E}_{\mathcal{S}}^{\circ}(n,1+x)$
of $\mathcal{S}$ enumerating the strictly order-preserving maps from
subposets of $\mathcal{S}$ to the interval $\{1,\ldots,n\}$:
\begin{equation}
\text{ZZ}(\boldsymbol{S},x)\equiv\text{E}_{\mathcal{S}}^{\circ}(n,1+x).\label{eq:equiv}
\end{equation}
The equivalence between both polynomials is demonstrated by Theorem~\ref{thm:PolyEqui},
which constitutes the main result of our paper.
\end{itemize}

\section{Preliminaries}

\subsection{Poset theory\label{sec:Glossary}}

The poset terminology used here follows closely Stanley's book \cite{stanley_enumerative_1986}.
A \emph{partially ordered set} $P$, or \emph{poset} for short, is
a set together with a binary relation $<_{P}$. In this manuscript,
we are concerned with finite posets $P$ with $p$ elements and with
strict partial orders, meaning that the relation $<_{P}$ is irreflexive,
transitive and antisymmetric. We say that the element $t\in P$ covers
the element $s\in P$ (denoted as $s\lessdot_{P}t$) if $s<_{P}t$
and there is no element $u\in P$ such that $s<_{P}u<_{P}t$. The
relation $<_{P}$ of a finite poset $P$ is entirely determined by
its cover relation, which allows us to represent $P$ graphically
in the form of a \emph{Hasse diagram}: The vertices of the Hasse diagram
are the elements of $P$, and every cover relation $s\lessdot_{P}t$
is represented by an edge that is drawn upwards from $s$ to $t$.
An \emph{induced subposet} $Q\subset P$ is a subset of $P$ together
with the order $<_{Q}$ inherited from $P$ which is defined for any
$s,t\in Q$ by $s<_{P}t\iff s<_{Q}t$. The usual symbol $<$ denotes
the relation ,,larger than'' in $\mathbb{N}$. The symbol $[\,n\,]$
stands for the set $\left\{ 1,2,\ldots,n\right\} $, and $[\,n,m\,]$
stands for the set $\left\{ n,n+1,\ldots,m\right\} $. The symbol
$\boldsymbol{n}$ represents the chain $1<2<3<...<n$. We say that
a map $\phi:P\rightarrow\mathbb{\mathbb{\mathbb{N}}}$ is \emph{order-preserving}
if it satisfies $s<_{P}t\Rightarrow\phi(s)\leq\phi(t)$, and \emph{strictly
order-preserving} if it satisfies $s<_{P}t\Rightarrow\phi(s)<\phi(t)$.
A \emph{natural labeling} of a poset $P$ is an order-preserving bijection
$\omega:P\rightarrow[\,p\,]$. A \emph{linear extension} of $P$ is
an order-preserving bijection $\sigma:P\rightarrow\boldsymbol{p}$,
which is often represented as a permutation $\omega\circ\sigma^{-1}$
expressed by the sequence $w=w_{1}w_{2}\ldots w_{p}$ with $w_{i}=\omega(\sigma{}^{-1}(i))$.
The set of all such linear extensions $w$ is denoted by $\mathcal{L}\left(P\right)$
and is usually referred to as the Jordan-Hölder set of $P$. If two
subsequent labels $w_{i}$ and $w_{i+1}$ in $w$ stand in the relation
$w_{i}>w_{i+1}$, then the index $i$ is called a \emph{descent} of
$w$. The total number of descents of $w$ is denoted by $\text{des}(w)$.

The strict order polynomial $\Omega_{P}^{\circ}(n)$ of a poset $P$
\cite{stanley_chromatic-like_1970,stanley_ordered_1972,stanley_enumerative_1986}
enumerates the strictly order-preserving maps $\phi:P\rightarrow[\,n\,]$
and can be expressed as
\begin{equation}
\Omega_{P}^{\circ}(n)=\sum_{w\in\mathcal{L}(P)}\binom{n+\text{des}(w)}{p}.\label{eq:StrictOrderPoly}
\end{equation}
The \emph{extended strict order polynomial $\text{\emph{E}}_{P}^{\circ}(n,z)$}
of a poset $P$ is formally defined \cite{langner2020sheep8} as
\begin{equation}
\text{E}_{P}^{\circ}(n,z)=\sum_{Q\subset P}\Omega_{Q}^{\circ}(n)z^{\left|Q\right|},\label{eq:GenOrderPoly}
\end{equation}
where the sum runs over the induced subposets $Q$ of $P$. The theorem
demonstrated recently by us (Theorem~2 of \cite{langner2020sheep8})
allows us to rewrite Eq.~(\ref{eq:GenOrderPoly}) in a more explicit
form
\begin{equation}
\text{E}_{P}^{\circ}(n,z)=\sum_{k=0}^{p}\sum_{w\in\mathcal{L}(P)}\binom{p-\text{fix}_{P}(w)}{k-\text{fix}_{P}(w)}\binom{n+\text{des}(w)}{k}z^{k},\label{eq:Znz}
\end{equation}
where $\text{fix}_{P}(w)$ denotes the number of fixed labels in the
linear extension $w=w_{1}w_{2}\ldots w_{p}$ of $P$. A label $w_{i}$
is \emph{fixed} in $w$ if at least one of the following two conditions
is satisfied: $(1)$~$i-1$ or $i$ is a descent, or $(2)$~the
set $L(w_{i})=\{l\,\vert\,l<i,w_{l}>w_{i}\}$ of positions of preceding
larger labels and the set $J(w_{i})=\{j\,\vert\,\omega^{-1}(w_{j})<_{P}\omega^{-1}(w_{i})\}$
of positions of necessarily preceding labels satisfy the following
two conditions: $L(w_{i})\ne\varnothing$ and $\text{max}(L(w_{i}))>\text{max}(J(w_{i}))$.

\subsection{Chemical graph theory}

There are many non-equivalent definitions of a benzenoid in the literature
\cite{gutman1989introduction,zhang1996theclar,Si2000,John1990}. In
the current paper, we aim at studying a specific family of benzenoids
(regular $m$-tier benzenoid strips), so for our purpose it is sufficient
to define a benzenoid $\boldsymbol{B}$ as a finite subgraph of the
infinite hexagonal lattice $L$, obtained by choosing a cycle $C_{\boldsymbol{B}}$
in $L$ and selecting all the vertices and edges of $L$ that lie
on or inside $C_{\boldsymbol{B}}$ \cite{zhang1996theclar}. We assume
that the lattice $L$ is oriented such that some of its edges are
vertical. We say that a benzenoid $\boldsymbol{B}$ is a \emph{regular
1-tier strip of length $n$} if it consists of $n$ adjacent hexagons
located in the same horizontal row of $L$. A \emph{regular $m$-tier
strip} $\boldsymbol{S}$ is obtained by merging $m$ consecutive regular
1-tier strips located in adjacent rows of $L$, in such a way that
the following two conditions are satisfied: $\left(i\right)$~Two
adjacent strips differ at each end by $\pm\sfrac{1}{2}$ hexagon unit.
$\left(ii\right)$~The top and the bottom regular 1-tier strips are
both of the same length $n$ \cite{witek2015zhangzhang,witek2017zhangzhang}.
Examples of regular $m$-tier strips are given in Figs.~\ref{fig:Regstrips},
\ref{fig2KofO332-1}, and \ref{fig2KofO332-1-1}.
\begin{figure}[H]
\begin{centering}
\includegraphics[scale=0.35]{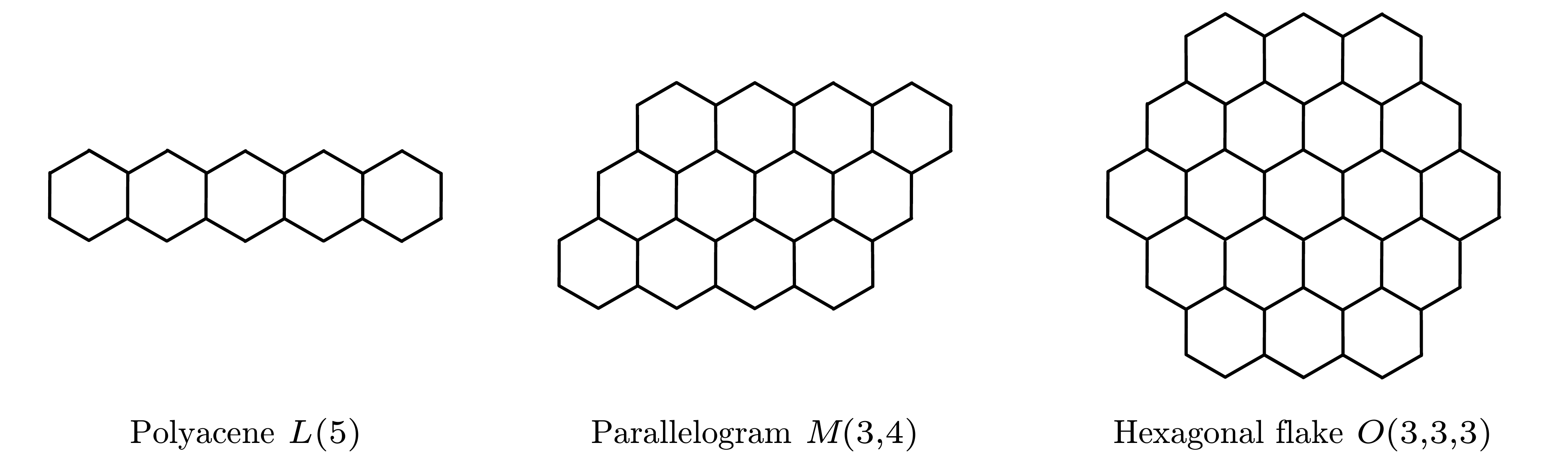}
\par\end{centering}
\caption{\label{fig:Regstrips}Examples of regular $m$-tier strips of length
$n$, with $m=1,3,5$ and $n=5,4,3$, respectively.}
\end{figure}

A \emph{Clar cover} is a spanning subgraph of $B$ such that every
one of its connected components is isomorphic to $K_{2}$ or $C_{6}$
\cite{clar1972thearomatic,pletervsek2016equivalence}. ($K_{2}$ denotes
a complete graph on 2 vertices and $C_{6}$ denotes a cycle of girth
6.) A spanning subgraph of $\boldsymbol{B}$ consisting entirely of
$K_{2}$ components is usually referred to as a \emph{Kekulé structure}
\cite{kekule1866untersuchungen,gutman1989introduction}, a perfect
matching, or a 1-factor of $\boldsymbol{B}$. Both concepts, Kekulé
structures and Clar covers, played very important roles in the early
development stages of theoretical chemistry, when it seemed plausible
that accurate predictions of energetic stability and reactivity of
benzenoid hydrocarbons could be directly linked to the theory of chemical
resonance \cite{Pauling1960resonance,clar1972thearomatic} based on
topological invariants derived by analyzing Kekulé structures and
Clar covers of a given benzenoid. Unfortunately, early quantum chemical
methods based on these concepts, such as the Hückel method or the
extended Hückel method, could not withstand the competition from much
more accurate and sophisticated computational models of quantum chemistry,
such as density functional theory and \emph{ab initio} methods, and
have been gradually sinking into oblivion. However, we see potential
capabilities lying dormant in those graph-theoretical concepts, which
might in the near future lead to a renaissance of Kekulé structure\textendash{}
or Clar cover\textendash based novel techniques of quantum chemistry,
using for example the set of Kekulé structures or Clar covers as a
Hilbert space basis for valence bond configuration interaction (CI)
or perturbation theory calculations. Such methods are yet to be developed,
but their latent advantages rely on the efficient enumeration of Clar
covers or Kekulé structures and generation of Hamiltonian matrix elements
in their basis using concepts similar to those used in the graphical
unitary group approach (GUGA) to CI \cite{Paldus1981,Shavitt1981}.
We hope that the research reported here will contribute to such a
development.

\begin{figure}[H]
\begin{centering}
\includegraphics[scale=0.3]{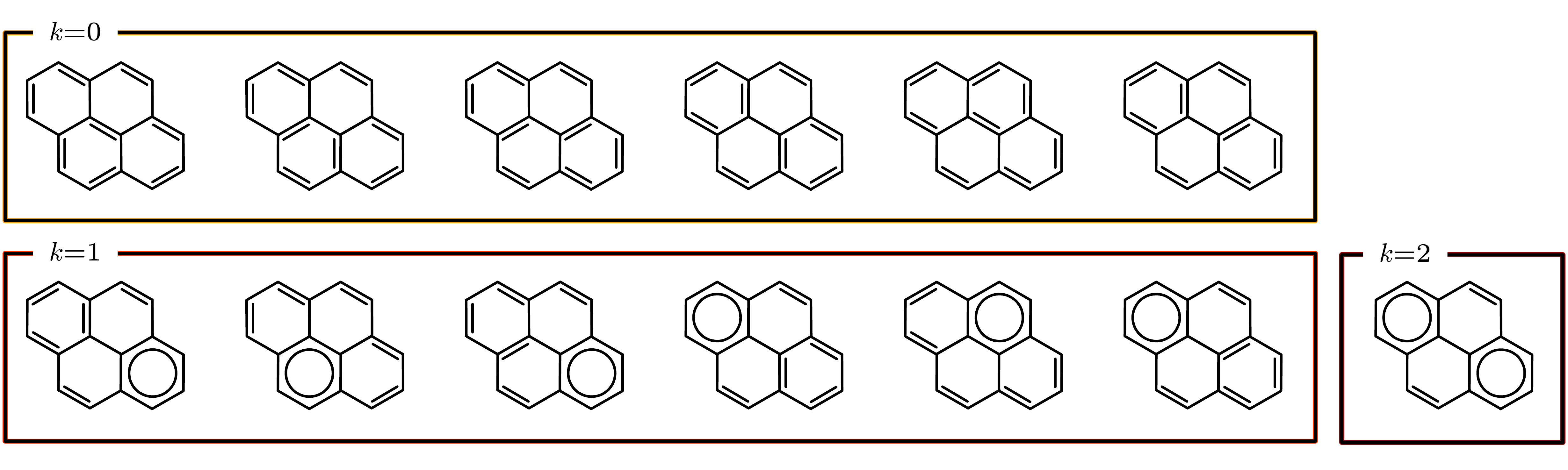}
\par\end{centering}
\caption{The parallelogram $M(2,2)$ has $6$ possible Clar covers with no
aromatic rings, $6$ Clar covers with one aromatic ring, and one Clar
cover with two aromatic rings. Therefore, $\text{ZZ}(M(2,2))=x^{2}+6x+6$.\label{fig:M22_ZZ}}
\end{figure}
A Clar cover $C$ of a benzenoid $\boldsymbol{B}$ with $N$ vertices
(atoms) consists of a certain number (say $k$) of hexagon ($C_{6}$)
components and $\frac{N-6k}{2}$ edge ($K_{2}$) components. The number
$k$ is referred to as the \emph{order} of the Clar cover $C$. As
an example, we present in Fig.~\ref{fig:M22_ZZ} all 13 possible
Clar covers of the parallelogram $M(2,2)$ (i.e., the only regular
2-tier benzenoid strip of length 2). This set consists of six Clar
covers of order 0 (coinciding with the Kekulé structures of $M(2,2)$),
six Clar covers of order 1 and a single Clar cover of order 2. The
maximal order $Cl$ of the Clar covers of $\boldsymbol{B}$, naturally
bounded from above by $\frac{N}{6}$, is referred to as the \emph{Clar
number} of $\boldsymbol{B}$ \cite{clar1972thearomatic,gutman1985clarformulas}.
For $M(2,2)$ from Fig.~\ref{fig:M22_ZZ}, we have $Cl=2$. Let us
denote by $c_{k}$ the number of Clar covers of $\boldsymbol{B}$
of order $k$. A generating function for the sequence $c_{0},\ldots,c_{Cl}$
\begin{eqnarray}
\text{ZZ}(\boldsymbol{B},x) & = & \sum_{k=0}^{Cl}c_{k}x^{k}\label{eq:ZZdef0}
\end{eqnarray}
was introduced by Zhang and Zhang as the \emph{Clar covering polynomial}
\cite{zhang1996theclar,zhang1996theclar2,zhang1997theclar,zhang2000theclar},\textit{\emph{
but is more commonly referred to in the modern literature as the}}\textit{
Zhang-Zhang polynomial} of $\boldsymbol{B}$ or simply the \textit{ZZ
polynomial} of $\boldsymbol{B}$. Zhang and Zhang demonstrated in
Theorem~2 of \cite{zhang2000theclar} (see also Theorem~1 of \cite{zhang1997theclar})
that the ZZ polynomial of a benzenoid $\boldsymbol{B}$ can also be
expressed as
\begin{equation}
\text{ZZ}(\boldsymbol{B},x)=\sum_{k=0}^{Cl}a(\boldsymbol{B},k)(x+1)^{k},\label{eq:ZZz}
\end{equation}
where $a(\boldsymbol{B},k)$ denotes the number of Kekulé structures
of $\boldsymbol{B}$ that have exactly $k$ proper sextets, where
a \emph{proper sextet} is characterized by three edges of a Kekulé
structure arranged within a hexagon as shown in the left panel of
Fig.~\ref{fig:propersextet}. The right panel of Fig.~\ref{fig:propersextet}
gives a few examples of other coverings of a single hexagon that do
not adhere to the definition of a proper sextet. Owing to Eq.~(\ref{eq:ZZz}),
one can compute the ZZ polynomial of a benzenoid $\boldsymbol{B}$
from the analysis of the set of Kekulé structures of $\boldsymbol{B}$,
which is substantially smaller than the set of Clar covers of $\boldsymbol{B}$.
We heavily rely on this concept in the following.
\begin{figure}[H]
\begin{centering}
\includegraphics[scale=0.3]{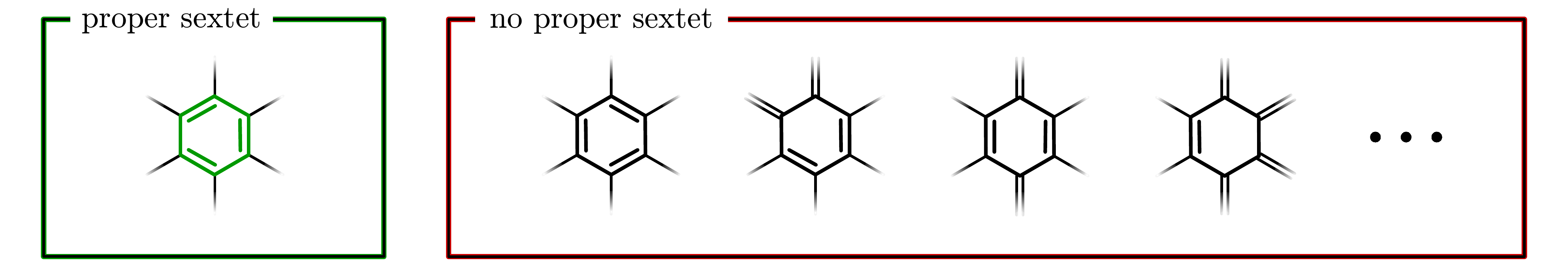}
\par\end{centering}
\caption{Left panel: Proper sextet. Right panel: Various coverings of a single
hexagon that are not proper sextets.\label{fig:propersextet}}
\end{figure}

\subsection{Adaptation of the interface theory of benzenoids to the analysis
of Kekulé structures of regular $m$-tier strips\label{sec:IT}}

Recently, we have developed a new theoretical framework for constructing,
analyzing, and enumerating Kekulé structures and Clar covers of benzenoids,
which is based on the concepts of fragments and interfaces. The resulting
conceptual methodology was given the name of \emph{interface theory
of benzenoids} \cite{langner2019IFTTheorems5,langner2019BasicApplications6}.
In this section, the main concepts and results of the interface theory
of benzenoids are presented in a simplified form specialized for studying
Kekulé structures of regular benzenoid strips. By virtue of Eq.~(\ref{eq:ZZz}),
the presented formalism is sufficient to enumerate Clar covers of
regular benzenoid strips and to compute their ZZ polynomials. In most
cases, the presented theory is obviously consistent with the previous
developments, but in situations when doubts might arise we give formal
proofs of the presented facts.

We introduce the following linguistic equivalences to be used throughout
this paper, which unify the terminologies typically used in the context
of graph theoretical analysis of benzenoids by the mathematical and
chemical communities: For a benzenoid $\boldsymbol{B}$ with a Kekulé
structure $K$, a vertex in $\boldsymbol{B}$ $\equiv$ an atom in
$\boldsymbol{B}$, an edge in $\boldsymbol{B}$ $\equiv$ a bond in
$\boldsymbol{B}$, an edge covered by some $K_{2}$ in $K$ $\equiv$
a double bond, and an edge in $\boldsymbol{B}$ that is not in $K$
$\equiv$ a single bond.

\begin{figure}[H]
\begin{centering}
\includegraphics[scale=0.4]{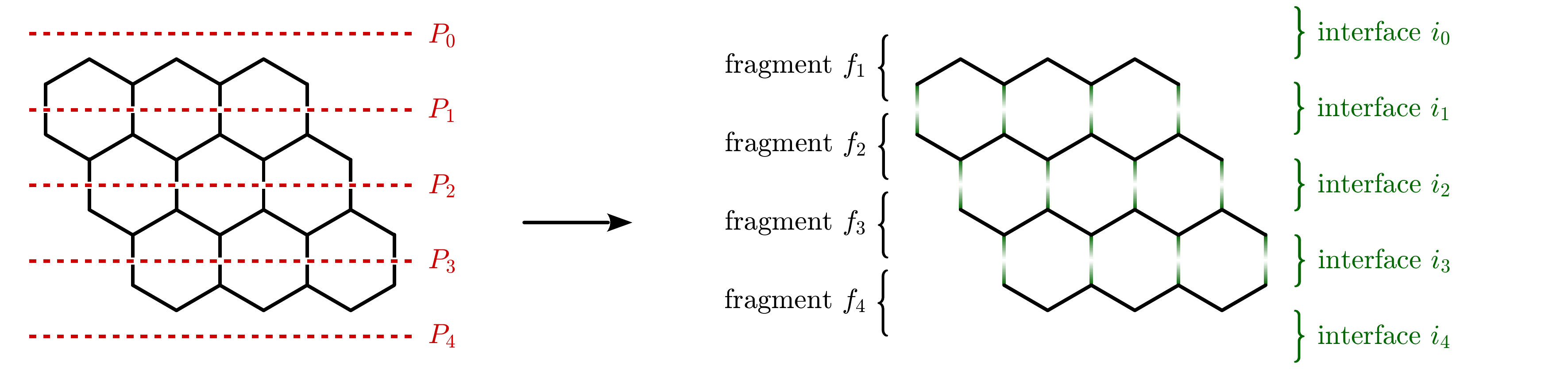}
\par\end{centering}
\caption{Dividing a regular $m$-tier strip of length $n$ using $m+2$ horizontal
partition lines $P_{0},\ldots,P_{m+1}$ defines $m+1$ fragments $f_{1},\ldots,f_{m+1}$
and $m+2$ interfaces $i_{0},\ldots,i_{m+1}$. These concepts are
illustrated here on the example of the parallelogram $M(m,n)$ with
$m=3$ and $n=3$.\label{fig:fragmentdef-1}}
\end{figure}
Consider a regular $m$-tier strip $\boldsymbol{S}$ with $m+2$ horizontal
partition lines $P_{0},\ldots,P_{m+1}$ along each row, as shown on
the left side of Fig.~\ref{fig:fragmentdef-1}. The vertical bonds
of $\boldsymbol{S}$ crossed by the partition lines are called \emph{interface
bonds}, and the slanted bonds between the partition lines are called
\emph{spine bonds}. The set of bonds and atoms which are (at least
partially) between the lines $P_{k-1}$ and $P_{k}$ is called the
\textit{fragment} $f_{k}$ of $\boldsymbol{B}$. The set of interface
bonds crossed by the line $P_{k}$ is called the \textit{interface
}$i_{k}$ of $\boldsymbol{B}$\textit{.}\textit{\emph{ The interfaces
}}$i_{k-1}$ and $i_{k}$ above and below the fragment $f_{k}$ are
called the \emph{upper and lower interfaces of $f_{k}$}, respectively.
Each fragment is assigned a shape, as depicted in Fig.~\ref{fig:Fragment-shapes.-1},
in the following way. 
\begin{figure}[H]
\begin{centering}
\includegraphics[scale=0.4]{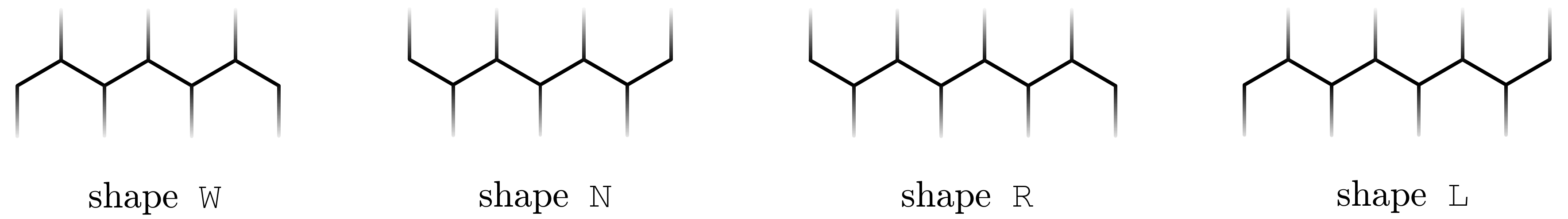}
\par\end{centering}
\caption{Fragments can have four possible shapes: $\mathtt{W}$ (wide), $\mathtt{N}$
(narrow), $\mathtt{R}$ (right), or $\mathtt{L}$ (left).\label{fig:Fragment-shapes.-1}}
\end{figure}
 Let $f_{k}$ be a fragment of a regular strip. Denote the leftmost
interface bond of $f_{k}$ by $b_{\text{first}}$ and the rightmost
interface bond of $f_{k}$ by $b_{\text{last}}$. The \textit{shape}
of $f_{k}$ is \label{def:shape-1}\vspace{-2mm}
\[
\begin{array}{lcll}
\mathtt{W} & \text{(wide)} & \text{if }b_{\text{first}}\in i_{k} & \text{and }b_{\text{last}}\in i_{k},\\
\mathtt{N} & \text{(narrow)} & \text{if }b_{\text{first}}\in i_{k-1} & \text{and }b_{\text{last}}\in i_{k-1},\\
\mathtt{R} & \text{(right)} & \text{if }b_{\text{first}}\in i_{k-1} & \text{and }b_{\text{last}}\in i_{k},\\
\mathtt{L} & \text{(left)} & \text{if }b_{\text{first}}\in i_{k} & \text{and }b_{\text{last}}\in i_{k-1},
\end{array}
\]
where $i_{k-1}$ is the upper interface of $f_{k}$ and $i_{k}$ is
the lower interface of $f_{k}$. For example, in the parallelogram
shown in Fig.~\ref{fig:fragmentdef-1}, the fragment $f_{1}$ has
shape $\mathtt{W}$, $f_{2}$ and $f_{3}$ have shape $\mathtt{R}$,
and $f_{4}$ has shape $\mathtt{N}$. It is clear that the sequence
$[\mathtt{W},\mathtt{R},\mathtt{R},\mathtt{N}]$ and the length $n=3$
fully specify the geometry of this particular regular strip.

\textit{\emph{The bonds within each interface are numbered from the
left to the right: The $j^{\text{th}}$ bond in the interface $i_{k}$
is denoted by $e_{k,j}$. Note that this differs from the bond notation
in the previous papers \cite{langner2019IFTTheorems5,langner2019BasicApplications6};
it is however the most convenient naming system for the following
derivations. The cardinality of a set $J$ of edges will be denoted
by $\left|J\right|$.}}
\begin{defn}
Consider a regular strip $\boldsymbol{S}$ and one of its Kekulé structures
$K$. The \emph{set} $K_{I}$\emph{ of double interface bonds} in
$K$ is defined as
\[
K_{I}=E\left(K\right)\cap\bigcup_{k=1}^{m}i_{k}.
\]
\end{defn}

\noindent Any Kekulé structure $K$ is uniquely determined by its
set $K_{I}$:
\begin{lem}[Lemma~8 of \cite{langner2019IFTTheorems5}]
\emph{\negthickspace{}\negthickspace{}\negthickspace{}}\label{lem:Fulldet}
~Let $\boldsymbol{S}$ be a regular strip, and consider two Kekulé
structures $K$ and $K'$ of $\boldsymbol{S}$. If their sets of double
interface bonds coincide, i.e., $K_{I}=K_{I}'$, then $K=K'$.
\end{lem}

\begin{proof}
The set $K_{I}$ specifies the covering character of all interface
bonds: those in $K_{I}$ are double bonds, and the remaining ones
are single bonds. According to Lemma~8 of \cite{langner2019IFTTheorems5},
the interface bond covering characters fully determine the entire
Clar cover (here $K$). Therefore, $K_{I}=K_{I}'$ implies $K=K'$.
\end{proof}
\begin{defn}
Consider a regular strip $\boldsymbol{S}$ and one of its Kekulé structures
$K$. The \textit{order }of an interface $i$ of $\boldsymbol{S}$
is defined as the number of edges of $K$ located in $i$:
\[
\text{ord}(i)=\left|K_{I}\cap i\right|.
\]
\end{defn}

\begin{thm}[First Rule: Interface order criterion (Theorem~11 of \cite{langner2019IFTTheorems5})]
\emph{\label{thm:1st rule-1}}Consider a regular $m$-tier strip
$\boldsymbol{S}$ of length $n$ with a Kekulé structure $K$. Let
$i$ be an interface of $\boldsymbol{S}$. Then, 
\begin{equation}
\text{ord}(i)=\left|i\right|-n.\label{eq:FirstRule}
\end{equation}
\end{thm}

\begin{proof}
According to Theorem~11 of \cite{langner2019IFTTheorems5} and the
definition of regular strips,\begin{footnotesize}
\[
\text{ord}(i_{k})=\begin{cases}
\text{ord}(i_{k-1})+1 & \text{ if shape}(f_{k})=\mathtt{W},\\
\text{ord}(i_{k-1})-1 & \text{ if shape}(f_{k})=\mathtt{N},\\
\text{ord}(i_{k-1}) & \text{ if shape}(f_{k})=\mathtt{R}\text{ or \ensuremath{\mathtt{L}}},
\end{cases}\quad\text{ and }\quad\left|i_{k}\right|=\begin{cases}
\left|i_{k-1}\right|+1 & \text{ if shape}(f_{k})=\mathtt{W},\\
\left|i_{k-1}\right|-1 & \text{ if shape}(f_{k})=\mathtt{N},\\
\left|i_{k-1}\right| & \text{ if shape}(f_{k})=\mathtt{R}\text{ or \ensuremath{\mathtt{L}}}.
\end{cases}
\]
\end{footnotesize}The empty interface $i_{0}$ naturally has $\text{ord}(i_{0})=0$.
The first fragment $f_{1}$ of a regular strip always has the shape
$\mathtt{W}$, thus the first interface $i_{1}$ has $\text{ord}(i_{1})=\text{ord}(i_{0})+1=1$.
In a regular strip of length $n$, the interface $i_{1}$ has $\left|i_{1}\right|=n+1$
bonds, and therefore $\left|i_{1}\right|-\text{ord}(i_{1})=n$. Owing
to the recursive properties of $\text{ord}(i_{k})$ and $\left|i_{k}\right|$
specified a few lines above, for every fragment shape, we find $\left|i_{k}\right|-\text{ord}(i_{k})=\left|i_{k-1}\right|-\text{ord}(i_{k-1})$.
It follows by induction that $\left|i_{k}\right|-\text{ord}(i_{k})=n$
for all $k=1,2,\ldots,m$.
\end{proof}
\begin{rem}
\label{rem:ordik}We see from Theorem~\ref{thm:1st rule-1} that
$\text{ord}(i)=\left|i\right|-n$, and since $\left|i\right|$ and
$n$ are both independent of $K$, so is $\text{ord}(i)$. This signifies
that the interface orders $\text{ord}(i_{1}),\ldots,\text{ord}(i_{m})$
are identical for all Kekulé structures $K$ of $\boldsymbol{S}$
and constitute yet another topological invariant of $\boldsymbol{S}$.
We will in the following take advantage of this fact, and understand
the order $\text{ord}(i)$ of an interface $i$ of $\boldsymbol{S}$
to be equal to $\left|i\right|-n$, even when no Kekulé structure
is specified.
\end{rem}

\begin{thm}[Second Rule: Double interface bonds alternate (Theorem~16 of \cite{langner2019IFTTheorems5})]
Let $\boldsymbol{S}$ be a regular strip with a Kekulé structure
$K$. Let $f$ be a fragment of $\boldsymbol{S}$. The double interface
bonds of $K$ belonging to $f$ (i.e., edges in $K_{I}\cap f$) are,
from left to right, distributed as follows.\label{thm:2nd rule-1}
\begin{lyxlist}{00.00.0000}
\item [{\emph{$(a)$}}] The first interface bond and the first double interface
bond in $f$ belong to the same interface.
\item [{\emph{$(b)$}}] The last interface bond and the last double interface
bond in $f$ belong to the same interface.
\item [{\emph{$(c)$}}] If there is a double bond in the upper (lower)
interface, then the next double bond can only belong to the lower
(upper) interface.
\end{lyxlist}
\end{thm}

\begin{proof}
Since a regular strip $\boldsymbol{S}$ is a benzenoid and since its
Kekulé structure $K$ is a Clar cover, all assumptions of Theorem~16
of \cite{langner2019IFTTheorems5} are satisfied and we know from
the statements $\left(a\right)$ and $\left(b\right)$ of the reformulation
of Theorem~16 of \cite{langner2019IFTTheorems5} given by Theorem~4
of \cite{langner2019BasicApplications6} that the statements $\left(a\right)$
and $\left(b\right)$ of the current theorem are true, because in
a regular strip $\boldsymbol{S}$ the first (last) atom of $f$ is
connected to the first (last) interface bond of $f$. The statement
$\left(c\right)$ of the current theorem is a direct consequence of
the statement $\left(c'\right)$ of Theorem~16 of \cite{langner2019IFTTheorems5}.
\end{proof}
\begin{thm}[Third Rule: Construction of Clar covers (Theorem~21 of \cite{langner2019IFTTheorems5})]
\label{thm:Construct_CC-1}Consider a regular $m$-tier strip $\boldsymbol{S}$.
Define a set $K_{v}$ of double interface bonds in $\boldsymbol{S}$
satisfying for every $k\in\left[\,m\,\right]$ the following conditions:
\begin{lyxlist}{00.00.0000}
\item [{\emph{$(a)$}}] $\left|K_{v}\cap i_{k}\right|=\text{ord}(i_{k})$ 
\item [{\emph{$(b)$}}] The set $K_{v}\cap f_{k}$ satisfies the statements
$\left(a\right)$\textendash $\left(c\right)$ of Theorem~\ref{thm:2nd rule-1}.
\end{lyxlist}
\noindent Then, there is exactly one Kekulé structure $K$ with $K_{I}=K_{v}$.
\end{thm}

\begin{proof}
The present theorem is a direct consequence of Theorem~5 of \cite{langner2019BasicApplications6}
applied to Kekulé structures (i.e., Clar covers of order 0) of regular
strips. Reinterpreting the set $K_{v}$ in the language used in \cite{langner2019BasicApplications6}
corresponds to assigning double bond covering character to all the
interface bonds in $K_{v}$, and single bond covering character to
all the interface bonds in $\bigcup_{k=1}^{m}i_{k}\setminus K_{v}$.
Condition $(a)$ of Theorem~5 of \cite{langner2019BasicApplications6}
is vacuously satisfied since none of the interface bonds in $\boldsymbol{S}$
have aromatic covering character. Condition $(a)$ of the present
Theorem implicitly defines interface orders $\text{ord}(i_{k})$ satisfying\textemdash by
the arguments presented in the proof of Theorem~\ref{thm:1st rule-1}
above\textemdash the conditions $(a)$\textendash $(c)$ of Theorem~3
of \cite{langner2019BasicApplications6}, which in turn implies the
validity of condition $(b)$ of Theorem~5 of \cite{langner2019BasicApplications6}.
The proof of Theorem~\ref{thm:2nd rule-1} shows that a set of double
interface bonds $K_{v}\cap f_{k}$ which satisfies condition $(b)$
of the present theorem also satisfies condition $(c)$ of Theorem~5
of \cite{langner2019BasicApplications6}. Therefore, all the conditions
of Theorem~5 of \cite{langner2019BasicApplications6} are satisfied,
meaning that there exists exactly one Clar cover with the double interface
bonds specified by $K_{v}$. Since none of the interface bonds are
aromatic, this unique Clar cover is a Kekulé structure $K$ with $K_{I}=K_{v}$. 
\end{proof}

\section{Derivation of the main results}

We always assume in the following that $\boldsymbol{S}$ is a regular
$m$-tier strip of length $n$ with at least one Kekulé structure.

\subsection{Partially ordered set of \soc s}

Consider an arbitrary Kekulé structure $K$ of the regular strip $\boldsymbol{S}$
and the corresponding set of double interface bonds $K_{I}$. We will
see in the following that the set $K_{I}$ can be naturally extended
to a poset, whose structure is completely determined by the First
and Second Rule of interface theory, and thus identical for every
Kekulé structure $K$.

Let us first analyze the structure of $K_{I}$. The number of bonds
of $K_{I}$ located in any interface $i_{k}$ of $\boldsymbol{S}$,
$\left|K_{I}\cap i_{k}\right|$, is, by Theorem~\ref{thm:1st rule-1},
equal to $\text{ord}(i_{k})=\left|i_{k}\right|-n$. Since $\left|i_{k}\right|-n$
is independent of $K$, the numbers of double bonds in the interfaces
$i_{1},\ldots,i_{m}$ are the same for every choice of $K$. Let us
refer to the $j^{th}$ element of $K_{I}\cap i_{k}$ from the left
as $d_{k,j}^{K}$. Then, we can express the set $K_{I}$ as
\begin{eqnarray}
K_{I} & = & \bigcup_{k=1}^{m}\bigcup_{j=1}^{\text{ord}(i_{k})}\left\{ d_{k,j}^{K}\right\} .\label{eq:KIunion}
\end{eqnarray}
It is clear that the structure of the set $K_{I}$ given by Eq.~(\ref{eq:KIunion})
is identical for every Kekulé structure $K$.

We can introduce now a strict partial order $<_{K}$ on the set $K_{I}$
via an appropriate cover relation $\lessdot_{K}$:
\begin{defn}
\label{def:relK}Consider two double interface bonds $d_{k,j}^{K},d_{k',j'}^{K}\in K_{I}\subset E\left(K\right)$
and denote $\kappa=\text{max}(k,k')$. We say that $d_{k,j}^{K}\text{ and }d_{k',j'}^{K}$
stand in the cover relation $d_{k,j}^{K}\lessdot_{K}d_{k',j'}^{K}$
if and only if
\begin{itemize}
\item $\left|k'-k\right|=1$ and
\item $j'-j=\begin{cases}
0 & \text{when the first interface bond of \ensuremath{f_{\kappa}} belongs to }i_{k},\\
1 & \text{when the first interface bond of \ensuremath{f_{\kappa}} belongs to }i_{k'}.
\end{cases}$
\end{itemize}
\end{defn}

\noindent The first condition, $\left|k'-k\right|=1$, indicates that
$d_{k,j}^{K}$ and $d_{k',j'}^{K}$ belong to the same fragment $f_{\kappa}$
of $\boldsymbol{S}$, and the second condition effectively stipulates
that $d_{k',j'}^{K}$ is the next double interface bond in $f_{\kappa}$
to the right from $d_{k,j}^{K}$:
\begin{thm}[\Soc~Reformulation of the Second Rule]
\label{thm:zigzagsoc}Consider a fragment $f$ and a Kekulé structure
$K$ of $\boldsymbol{S}$. Consider further two double interface bonds
$d_{k,j}^{K},d_{k',j'}^{K}\in K_{I}\cap f$ which stand in the relation
$d_{k,j}^{K}\lessdot_{K}d_{k',j'}^{K}$. Then, $d_{k,j}^{K}$ is located
to the left of $d_{k',j'}^{K}$, and all the interface bonds of $f$
between $d_{k,j}^{K}$ and $d_{k',j'}^{K}$ are single bonds.
\end{thm}

\begin{proof}
Denote by $i_{l}$ the interface of $f$ which contains the first
interface bond of $f$, and by $i_{r}$ the other interface of $f$.
According to condition $(a)$ of the Second Rule (given in Theorem~\ref{thm:2nd rule-1}),
the first double interface bond $d_{l,1}^{K}$ of $f$ is in $i_{l}$.
According to condition $(c)$ of Theorem~\ref{thm:2nd rule-1}, a
double bond in $i_{l}$ ($i_{r}$) is followed by a double bond in
$i_{r}$ ($i_{l}$). Therefore, the double interface bonds of $f$
are alternating between the interfaces $i_{l}$ and $i_{r}$ and are
given, from left to right, by the sequence $d_{l,1}^{K},d_{r,1}^{K},d_{l,2}^{K},d_{r,2}^{K},d_{l,3}^{K},\ldots$

\noindent For two double interface bonds of the fragment $f$ standing
in the relation $d_{k,j}^{K}\lessdot_{K}d_{k',j'}^{K}$, we need to
consider two possibilities: $\left(i\right)$~$k'=r$ or $\left(ii\right)$~$k'=l$.
In case~$\left(i\right)$, the first condition of Definition~\ref{def:relK}
tells us that $k=l$ and the second condition of Definition~\ref{def:relK}
tells us that $j'=j$. A comparison with the sequence specified at
the beginning of this proof shows that $d_{k,j}^{K}=d_{l,j'}^{K}$
is indeed located to the left of $d_{k',j'}^{K}=d_{r,j'}^{K}$. Similarly,
in case~$\left(ii\right)$, we have $k=r$ and $j'=j+1$. Again,
a comparison with the sequence specified at the beginning of this
proof shows that $d_{k,j}^{K}=d_{r,j'-1}^{K}$ is indeed located to
the left of $d_{k',j'}^{K}=d_{l,j'}^{K}$. Since $d_{l,j'}^{K}$ and
$d_{r,j'}^{K}$ in case~$\left(i\right)$ and $d_{r,j'-1}^{K}$ and
$d_{l,j'}^{K}$ in case~$\left(ii\right)$ are consecutive pairs
of double interface bonds in the above sequence, no double interface
bonds are located between them and all interface bonds (if any) located
between $d_{k,j}^{K}$ and $d_{k',j'}^{K}$ are single bonds.
\end{proof}
\begin{defn}
\label{def:Transitive-closure}The transitive closure of the relation
$\lessdot_{K}$ shall be denoted by $<_{K}$.
\end{defn}

\begin{lem}
The relation $<_{K}$ is a strict partial order.
\end{lem}

\begin{proof}
We have to show that $<_{K}$ is irreflexive, transitive and antisymmetric.
Transitivity is clear from Def.~\ref{def:Transitive-closure}. It
follows from Theorem~\ref{thm:zigzagsoc} that, whenever $d_{k,j}^{K}<_{K}d_{k',j'}^{K}$,
the double interface bond $d_{k,j}^{K}$ is located to the left of
$d_{k',j'}^{K}$; consequently, $<_{K}$ is antisymmetric and irreflexive.
\end{proof}
\begin{fact}
\label{fact:coverf}Consider a fragment $f$ of $\boldsymbol{S}$
and its Kekulé structure $K$. Denote by $i_{u}$ the upper interface
of $f$ and by $i_{l}$ the lower interface of $f$. It follows from
Def.~\ref{def:relK} that the set $\mathscr{C}_{K}\left(f\right)$
of cover relations $\lessdot_{K}$ between the double interface bonds
of $f$ (i.e., between the elements of $K_{I}\cap\left(i_{u}\cup i_{l}\right)=\left\{ d_{u,1}^{K},\ldots,d_{u,\text{ord}(i_{u})}^{K},d_{l,1}^{K},\ldots,d_{l,\text{ord}(i_{l})}^{K}\right\} $)
is completely specified by the following chain of inequalities
\begin{equation}
\begin{cases}
d_{u,1}^{K}\lessdot_{K}d_{l,1}^{K}\lessdot_{K}d_{u,2}^{K}\lessdot_{K}d_{l,2}^{K}\lessdot_{K}...\lessdot_{K}d_{l,\text{ord}(i_{l})}^{K}\lessdot_{K}d_{u,\text{ord}(i_{u})}^{K} & \text{ ~if shape}(f)=\mathtt{W}\\
d_{l,1}^{K}\lessdot_{K}d_{u,1}^{K}\lessdot_{K}d_{l,2}^{K}\lessdot_{K}d_{u,2}^{K}\lessdot_{K}...\lessdot_{K}d_{u,\text{ord}(i_{u})}^{K}\lessdot_{K}d_{l,\text{ord}(i_{l})}^{K} & \text{ ~if shape}(f)=\mathtt{N}\\
d_{l,1}^{K}\lessdot_{K}d_{u,1}^{K}\lessdot_{K}d_{l,2}^{K}\lessdot_{K}d_{u,2}^{K}\lessdot_{K}...\lessdot_{K}d_{l,\text{ord}(i_{l})}^{K}\lessdot_{K}d_{u,\text{ord}(i_{u})}^{K} & \text{ ~if shape}(f)=\mathtt{R}\\
d_{u,1}^{K}\lessdot_{K}d_{l,1}^{K}\lessdot_{K}d_{u,2}^{K}\lessdot_{K}d_{l,2}^{K}\lessdot_{K}...\lessdot_{K}d_{u,\text{ord}(i_{u})}^{K}\lessdot_{K}d_{l,\text{ord}(i_{l})}^{K} & \text{ ~if shape}(f)=\mathtt{L}
\end{cases}\label{eq:coverf}
\end{equation}
The structure of the set $\mathscr{C}_{K}\left(f\right)$ is independent
of the Kekulé structure $K$ used for its construction and depends
only on the structural parameters of $\boldsymbol{S}$: the shape
of a given fragment $f$ of $\boldsymbol{S}$ is obviously independent
of $K$ and the orders of its both interfaces, $\text{ord}(i_{u})$
and $\text{ord}(i_{l})$, are independent of $K$ by Remark~\ref{rem:ordik}. 
\end{fact}

\begin{fact}
\label{fact:coverS}Consider a regular benzenoid strip $\boldsymbol{S}$
and its Kekulé structure $K$. The complete set $\mathscr{C}_{K}\left(\boldsymbol{S}\right)$
of cover relations $\lessdot_{K}$ that can be constructed for $\boldsymbol{S}$
is given by
\begin{equation}
\mathscr{C}_{K}\left(\boldsymbol{S}\right)=\bigcup_{k=2}^{m}\mathscr{C}_{K}\left(f_{k}\right)\label{eq:coverS}
\end{equation}
Again, following the discussion at the end of Fact~\ref{fact:coverf},
the structure of the set $\mathscr{C}_{K}\left(\boldsymbol{S}\right)$
given by Eq.~(\ref{eq:coverS}) is identical for every Kekulé structure
$K$ and thus independent of $K$. Note that the fragments $f_{1}$
and $f_{m+1}$ can be excluded from the sum in Eq.~(\ref{eq:coverS}),
because they contain only one non-empty interface each and consequently
do not contribute any cover relations to $\mathscr{C}_{K}\left(\boldsymbol{S}\right)$.
\end{fact}

\noindent Below, in Examples~\ref{exa:M32poset} and~\ref{exa:O332poset},
we construct the Hasse diagram corresponding to the poset generated
by the complete set of cover relations $\mathscr{C}_{K}\left(\boldsymbol{S}\right)$
for two regular strips, $\boldsymbol{S}=M\left(3,2\right)$ and $\boldsymbol{S}=O\left(3,2,3\right)$.
\begin{example}
\label{exa:M32poset}Consider the two distinct Kekulé structures of
the parallelogram $M\left(3,2\right)$ shown in Figures~\ref{fig2KofM32}$\left(\text{a}\right)$
and $\left(\text{b}\right)$. Each interface of $M\left(3,2\right)$,
$i_{1}$ and $i_{2}$, has order one and contains one double bond,
denoted as $d_{1,1}^{K}$ and $d_{2,1}^{K}$, respectively. The set
$K_{I}$ thus contains two elements, $K_{I}=\left\{ d_{1,1}^{K},d_{2,1}^{K}\right\} $.
Definition~\ref{def:relK} allows us to establish one cover relation,
$d_{1,1}^{K}\lessdot_{K}d_{2,1}^{K}$, which is identical for each
selected Kekulé structure\textemdash in fact, since Definition~\ref{def:relK}
relies only on the indices of the elements of $K_{I}$, this cover
relation holds for every Kekulé structure of $M\left(3,2\right)$.
According to Theorem~\ref{thm:zigzagsoc}, the relation $d_{1,1}^{K}\lessdot_{K}d_{2,1}^{K}$
implies that $d_{1,1}^{K}$ is located to the left of $d_{2,1}^{K}$;
this is easily verified for any given Kekulé structure such as the
ones given in Fig.~\ref{fig2KofM32}$\left(\text{a}\right)$ and
$\left(\text{b}\right)$. The resulting Hasse diagram for $M\left(3,2\right)$
for the posets $K_{I}$ with the derived relation $<_{K}$ is shown
in Figure~\ref{fig2KofM32}$\left(\text{c}\right)$.
\begin{figure}[H]
\begin{centering}
\includegraphics[scale=0.8]{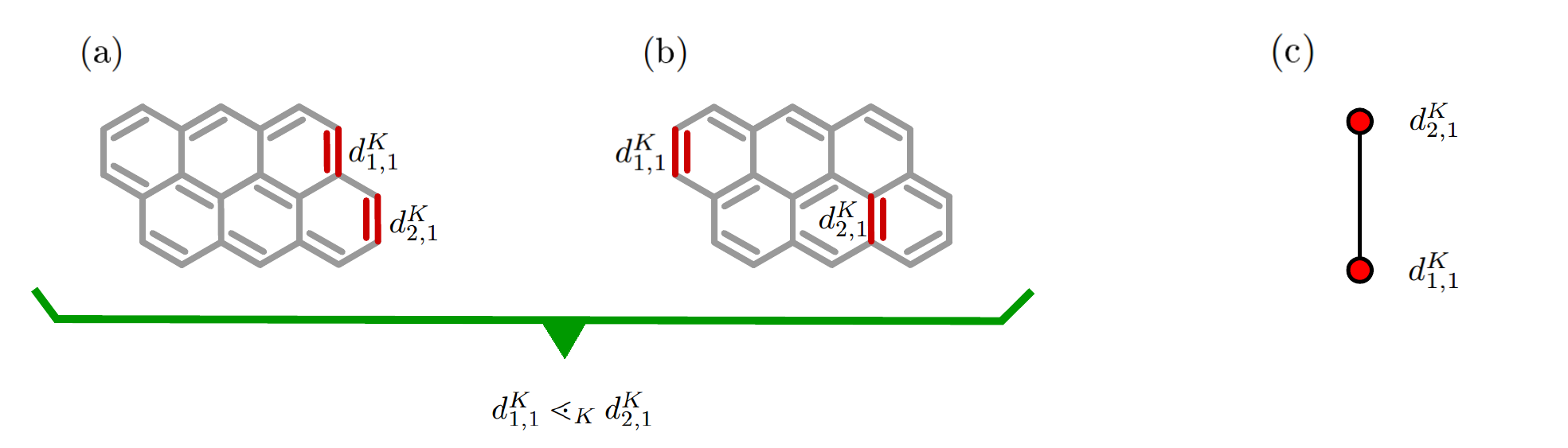}
\par\end{centering}
\caption{The double interface bonds $d_{1,1}^{K}$ and $d_{2,1}^{K}$ in two
distinct Kekulé structures $\left(\text{a}\right)$ and $\left(\text{b}\right)$
of the parallelogram $M\left(3,2\right)$ are related by the same
cover relation $d_{1,1}^{K}\lessdot_{K}d_{2,1}^{K}$, which indicates
that $d_{1,1}^{K}$ must be located to the left of $d_{2,1}^{K}$.
In $\left(\text{c}\right)$ we show the resulting Hasse diagram (identical
for every Kekulé structure of $M\left(3,2\right)$) corresponding
to the cover relation $d_{1,1}^{K}\lessdot_{K}d_{2,1}^{K}$. \label{fig2KofM32}}
\end{figure}
\end{example}

\begin{example}
\label{exa:O332poset}Consider two distinct Kekulé structures of the
hexagonal graphene flake $O\left(3,2,3\right)$ shown in Figures~\ref{fig2KofO332}$\left(\text{a}\right)$
and $\left(\text{b}\right)$. The orders of the interfaces $i_{1},\ldots,i_{4}$
are $1,2,2,1$, respectively. The set $K_{I}$ contains thus six elements,
$K_{I}=\left\{ d_{1,1}^{K},d_{2,1}^{K},d_{2,2}^{K},d_{3,1}^{K},d_{3,2}^{K},d_{4,1}^{K}\right\} $.
Definition~\ref{def:relK} allows us to establish seven cover relations
(for the detailed list, see Figure~\ref{fig2KofO332}). It can be
verified that these cover relations are identical for every Kekulé
structure of $O\left(3,2,3\right)$, and that Theorem~\ref{thm:zigzagsoc}
holds for each of these Kekulé structures. The Hasse diagram corresponding
to the posets $K_{I}$ of $O\left(3,2,3\right)$ is shown in Figure~\ref{fig2KofO332}$\left(\text{c}\right)$.
\begin{figure}[H]
\begin{centering}
\includegraphics[scale=0.6]{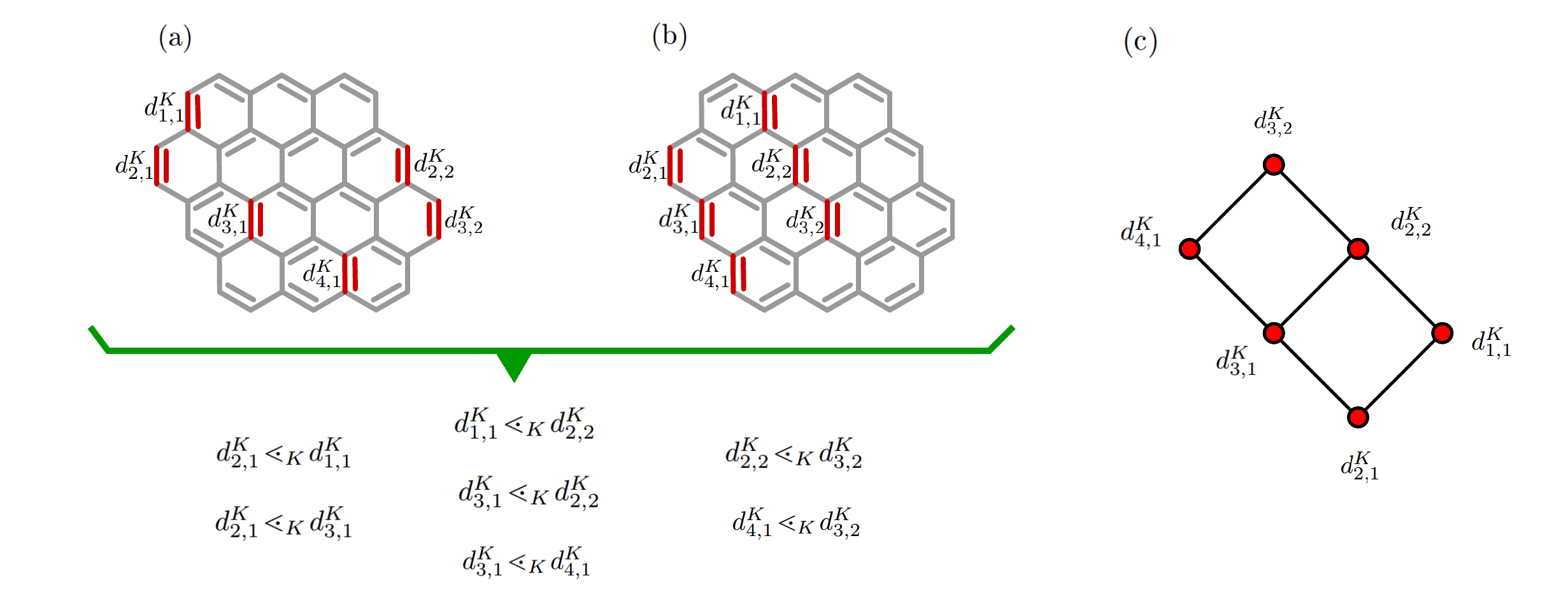}
\par\end{centering}
\caption{$\left(\text{a}\right)$ and $\left(\text{b}\right)$ Two distinct
Kekulé structures of the hexagonal graphene flake $O\left(3,2,3\right)$.
The cover relations determined from Definition~\ref{def:relK} are
the same for both Kekulé structures. $\left(\text{c}\right)$ The
resulting Hasse diagram for the partial order $<_{K}$ induced on
the set $K_{I}=\left\{ d_{1,1}^{K},d_{2,1}^{K},d_{2,2}^{K},d_{3,1}^{K},d_{3,2}^{K},d_{4,1}^{K}\right\} $
by the seven listed cover relations is identical for every Kekulé
structure of $O\left(3,2,3\right)$. \label{fig2KofO332}}
\end{figure}
\end{example}

We have demonstrated before that the structure of the set $K_{I}$
given by Eq.~(\ref{eq:KIunion}) is identical for every Kekulé structure
$K$ of $\boldsymbol{S}$. Similarly, we have demonstrated that the
structure of the set of cover relations $\mathscr{C}_{K}\left(\boldsymbol{S}\right)$
given by Eq.~(\ref{eq:coverS}) is identical for every Kekulé structure
$K$ of $\boldsymbol{S}$. This shows that the corresponding Hasse
diagrams (for example those constructed in Examples~\ref{exa:M32poset}
and~\ref{exa:O332poset}) are also independent of the choice of Kekulé
structure $K$ used for their construction; the only difference between
two Hasse diagrams constructed using two distinct Kekulé structures
is that the corresponding poset elements (i.e., the Hasse diagram
vertices) $d_{k,j}^{K}$ and $d_{k,j}^{K'}$ may stand for different
interface edges of $\boldsymbol{S}$. This notion can be further formalized
by introducing a poset $\mathcal{S}$ which is isomorphic to all the
posets $K_{I}$ of $\boldsymbol{S}$.
\begin{defn}
\label{def:PosetS}Consider a regular benzenoid strip $\boldsymbol{S}$
with $m$ interfaces $i_{1},\ldots,i_{m}$ of orders\linebreak{}
$\text{ord}(i_{1}),\ldots,\text{ord}(i_{m})$, respectively. We define
the poset $\mathcal{S}$ as
\begin{eqnarray*}
\mathcal{S} & = & \bigcup_{k=1}^{m}\bigcup_{j=1}^{\text{ord}(i_{k})}\left\{ s_{k,j}\right\} 
\end{eqnarray*}
 together with the partial order $<_{\mathcal{S}}$ defined via its
cover relations: Two elements $s_{k,j},s_{k',j'}\in\mathcal{S}$ stand
in the cover relation $s_{k,j}\lessdot_{\mathcal{S}}s_{k',j'}$ if
and only if
\begin{itemize}
\item $\left|k'-k\right|$$=1$ and
\item $j'-j=\begin{cases}
0 & \text{when the first interface bond of \ensuremath{f_{\kappa}} belongs to }i_{k},\\
1 & \text{when the first interface bond of \ensuremath{f_{\kappa}} belongs to }i_{k'},
\end{cases}$
\end{itemize}
\noindent where $\kappa=\text{max}(k,k')$. The elements $s_{k,j}$
of $\mathcal{S}$ are called the \emph{double interface bonds of $\boldsymbol{S}$},
or \emph{\soc s} for short. The set $\mathcal{S}$ will be referred
to as the \emph{set of \soc s.}
\end{defn}

\noindent Examples of posets $\mathcal{S}$ for nine classes of regular
benzenoid strips are shown in Figure~\ref{fig2KofO332-1}. 
\begin{figure}[H]
\begin{centering}
\includegraphics[scale=0.35]{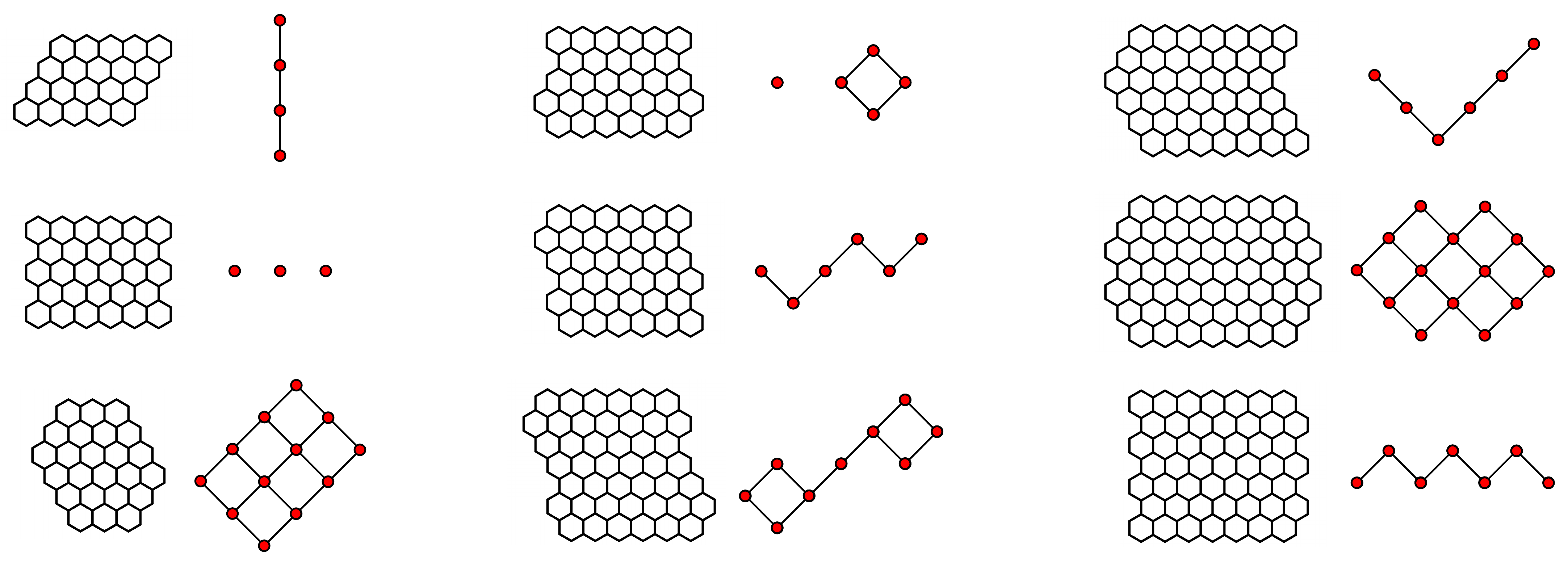}
\par\end{centering}
\caption{Examples of posets $\mathcal{S}$ constructed for nine selected regular
benzenoid strips $\boldsymbol{S}$. \label{fig2KofO332-1}}
\end{figure}

\noindent 
\begin{figure}[H]
\begin{centering}
\includegraphics[scale=0.45]{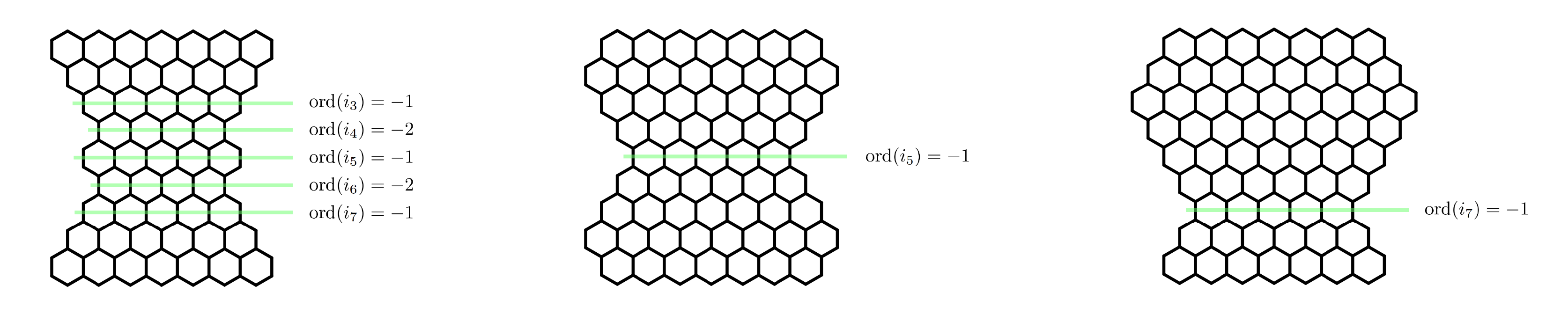}
\par\end{centering}
\caption{Not every regular strip $\boldsymbol{S}$ can be associated with a
poset $\mathcal{S}$. The three regular strips shown above contain
at least one interface with a negative order (marked in green online).
Consequently, by Example~11 of \cite{langner2019BasicApplications6},
all the three structures are non-Kekuléan, there is no poset associated
with them, and their ZZ polynomials are identically equal to 0. \label{fig2KofO332-1-1}}
\end{figure}

It is now straightforward to construct an obvious isomorphism $\tau_{K}$
between $\mathcal{S}$ and any $K_{I}$:
\begin{defn}
Consider a Kekulé structure $K$ of $\boldsymbol{S}$. Then the map
$\tau_{K}:\mathcal{S}\rightarrow K_{I}\subset E\left(K\right)$ is
defined by
\[
\tau_{K}:s_{k,j}\mapsto d_{k,j}^{K}\text{ for every }k\in[\,m\,],j\in[\,\text{ord}(i_{k})\,].
\]
\end{defn}

\begin{lem}
\label{<SK}If two \soc s stand in the relation $s_{k,j}\lessdot_{\mathcal{S}}s_{k',j'}$,
then for any Kekulé structure $K$ of $\boldsymbol{S}$ we have $\tau_{K}(s_{k,j})\lessdot_{K}\tau_{K}(s_{k',j'})$.
\end{lem}

\begin{proof}
This fact is a direct consequence of the similarity between Definitions~\ref{def:relK}
and \ref{def:PosetS} of the relations $\lessdot_{K}$ and $\lessdot_{\mathcal{S}}$.
If $s_{k,j}\lessdot_{\mathcal{S}}s_{k',j'}$, then the indices $k$,
$k'$, $j$ and $j'$ satisfy the two conditions in Def.~\ref{def:PosetS},
which are identical to the two conditions in Def.~\ref{def:relK},
meaning that $\tau_{K}(s_{k,j})=d_{k,j}^{K}\lessdot_{K}d_{k',j'}^{K}=\tau_{K}(s_{k',j'})$.
\end{proof}
\begin{rem}
Before continuing the exposition, let us summarize briefly the structure
of the studied regular benzenoid strip $\boldsymbol{S}$. We recollect
that $\boldsymbol{S}$ contains $m$ interfaces $i_{1},\ldots,i_{m}$.
The interface $i_{k}$ consists of $n+\text{ord}(i_{k})$ edges, out
of which $n$ have single covering character and $\text{ord}(i_{k})$
have double covering character in every Kekulé structure $K$ of $\boldsymbol{S}$.
\textit{\emph{The $j^{\text{th}}$ bond in the interface $i_{k}$
(with $j\in[\,n+\text{ord}(i_{k})\,]$) is denoted by $e_{k,j}$,
regardless of its covering character. Additionally, for each }}Kekulé
structure $K$, the \textit{\emph{$j^{\text{th}}$}} double interface
bond \textit{\emph{in the interface $i_{k}$ (with $j\in[\,\text{ord}(i_{k})\,]$)}}
is denoted by $d_{k,j}^{K}$. Consequently, for each $d_{k,j}^{K}$
there exists an index $\tilde{j}=\tilde{j}\left(K,k,j\right)\in[\,n+\text{ord}(i_{k})\,]$
such that $d_{k,j}^{K}\equiv e_{k,\tilde{j}}$. Combining this fact
with the definition of $\tau_{K}$, we can see that for each $s_{k,j}$
there exists an index $\tilde{j}=\tilde{j}\left(K,k,j\right)\in[\,n+\text{ord}(i_{k})\,]$
such that $\tau_{K}(s_{k,j})=e_{k,\tilde{j}}$. We use this fact in
Definition~\ref{def:posK} to define a new ($K$-dependent) property
of $s_{k,j}$.
\end{rem}

\begin{defn}
\label{def:posK}Let $\text{pos}_{K}:\mathcal{S}\rightarrow\mathbb{N}$
be the map which assigns to each $s_{k,j}$ a number $\text{pos}_{K}\left(s_{k,j}\right)$
such that $\tau_{K}(s_{k,j})=e_{k,\text{pos}_{K}\left(s_{k,j}\right)}$.
This number $\text{pos}_{K}\left(s_{k,j}\right)$ is called the \emph{position}
of $s_{k,j}$ in $K$.
\end{defn}

It is now clear that we can define any Kekulé structure $K$ of $\boldsymbol{S}$
simply by specifying $\mathcal{S}$ and the corresponding map $\text{pos}_{K}\left(s_{k,j}\right)$:
The set $K_{I}=\left\{ e_{k,\text{pos}_{K}\left(s_{k,j}\right)}\,\vert\,s_{k,j}\in\mathcal{S}\right\} $
uniquely determines $K$ by Lemma~\ref{lem:Fulldet}. 

\subsection{Equivalence between the extended strict order polynomial and the
ZZ polynomial}

We have seen in the previous subsection that the set $\mathcal{S}$
together with the relation $<_{\mathcal{S}}$ constitutes a partially
ordered set. The pair $\left(\mathcal{S},<_{\mathcal{S}}\right)$\textemdash denoted
concisely also as $\mathcal{S}$\textemdash encodes the relative positions
of double interface bonds within all Kekulé structures of a given
regular strip. The absolute positions of the elements of $K_{I}$
in a given Kekulé structure $K$ are encoded by the triple $\left(\mathcal{S},<_{\mathcal{S}},\text{pos}_{K}\right)$.
We will show in the following that, in order to define $K$, it is
sufficient to specify $\text{pos}_{K}\left(s_{k,j}\right)$ not for
all elements of $\mathcal{S}$, but only for the elements of the subposet
$\mathcal{A}_{K}\subset\mathcal{S}$ defined as follows.
\begin{defn}
Let $K$ be a Kekulé structure of $\boldsymbol{S}$. Then, we denote
by $\mathcal{A}_{K}$ the induced subposet of $\mathcal{S}$ given
by 
\[
\mathcal{A}_{K}=\left\{ s_{k,j}\in\mathcal{S}\,\vert\,\text{\ensuremath{\tau_{K}}\ensuremath{\left(s_{k,j}\right)} participates in a proper sextet of \ensuremath{K}}\right\} .
\]
\end{defn}

We will in the following Lemmata~\ref{thm:Deconstruct_CC} and \ref{thm:Construct_CC}
establish a one-to-one correspondence between Kekulé structures of
$\boldsymbol{S}$ and certain order-preserving maps from the elements
of subposets $\mathcal{A}_{K}\subset\mathcal{S}$ to the numbers encoding
the absolute positions of \soc s in $K$. This will allow us to establish,
in Theorem~\ref{thm:PolyEqui}, the equivalence between the ZZ polynomial
$\text{ZZ}(\boldsymbol{S},x)$ of a Kekuléan regular $m$-tier strip
$\boldsymbol{S}$ and the extended strict order polynomial $\text{E}_{\mathcal{S}}^{\circ}(n,x+1)$
of the corresponding poset $\mathcal{S}$.
\begin{lem}
\label{thm:Deconstruct_CC}Let $\boldsymbol{S}$ be a regular $m$-tier
strip of length $n$ with the poset of \soc s $\mathcal{S}$, and
let $K$ be a Kekulé structure of $\boldsymbol{S}$. Then, there is
exactly one strictly order-preserving map $\mu:\mathcal{A}_{K}\rightarrow\left[\,n\,\right]$
such that
\begin{enumerate}
\item $\text{\emph{pos}}_{K}\left(s_{k,j}\right)=\mu\left(s_{k,j}\right)+j$
for every $s_{k,j}\in\mathcal{A}_{K}$ and
\item $\text{\emph{pos}}_{K}\left(s_{k,j}\right)=\text{\emph{max}}\left(\left\{ \mu\left(s_{k',j'}\right)\,\vert\,s_{k',j'}\in\mathcal{A}_{K},s_{k',j'}<_{\mathcal{S}}s_{k,j}\right\} \cup\left\{ 0\right\} \right)+j$
\\
$\text{for every }s_{k,j}\in\mathcal{S}\setminus\mathcal{A}_{K}.$
\end{enumerate}
\end{lem}

\begin{lem}
\label{thm:Construct_CC}Let $\boldsymbol{S}$ be a regular $m$-tier
strip of length $n$ with the poset of \soc s $\mathcal{S}$. Let
further $\mathcal{A}\subset\mathcal{S}$ be an induced subposet of
$\mathcal{S}$, and let $\mu:\mathcal{A}\rightarrow\left[\,n\,\right]$
be a strictly order-preserving map. Then, there exists exactly one
Kekulé structure $K$ such that
\begin{enumerate}
\item $\text{\emph{pos}}_{K}\left(s_{k,j}\right)=\mu\left(s_{k,j}\right)+j$
for every $s_{k,j}\in\mathcal{A}$,
\item $\text{\emph{pos}}_{K}\left(s_{k,j}\right)=\text{\emph{max}}\left(\left\{ \mu\left(s_{k',j'}\right)\,\vert\,s_{k',j'}\in\mathcal{A},s_{k',j'}<_{\mathcal{S}}s_{k,j}\right\} \cup\left\{ 0\right\} \right)+j$\\
 $\text{for every }s_{k,j}\in\mathcal{S}\setminus\mathcal{A},$
\item $\mathcal{A}_{K}=\mathcal{A}$.
\end{enumerate}
\end{lem}

\noindent The proofs of Lemmata~\ref{thm:Deconstruct_CC} and\ \ref{thm:Construct_CC},
being rather long and technical, are exiled to the Appendix.

Lemmata~\ref{thm:Deconstruct_CC} and~\ref{thm:Construct_CC} are
complementary. Together they establish a one-to-one correspondence
between the set $\left\{ K\right\} $ of Kekulé structures of $\boldsymbol{S}$
and the set $\left\{ \mu:\mathcal{S}\supset\mathcal{A}\rightarrow\left[\,n\,\right]\right\} $
of strictly order-preserving maps from the induced subposets of $\mathcal{S}$
and the interval $\left[\thinspace n\thinspace\right]$. In particular,
the pair $\left(\mathcal{A},\mu\right)$ completely determines the
corresponding Kekulé structure $K$ in the following way. The cardinality
of the set $\mathcal{A}$ determines the number of proper sextets
in $K$; the positions of double interface bonds participating in
the proper sextets are fully and uniquely determined by the map $\mu$
using condition $1$ of Lemma\ \ref{thm:Construct_CC}. The positions
of the double interface bonds not participating in proper sextets
are fully and uniquely determined by the map $\mu$ using condition
$2$ of Lemma\ \ref{thm:Construct_CC}. Both sets of double interface
bonds define the set $K_{I}$ and\textemdash by Lemma~\ref{lem:Fulldet}\textemdash uniquely
determine the corresponding Kekulé structure $K$. This one-to-one
correspondence is used to conclude our investigations with
\begin{thm}
\label{thm:PolyEqui}Let $\boldsymbol{S}$ be a regular $m$-tier
strip of length $n$ with the poset of \soc s $\mathcal{S}$, and
consider an integer $k$. Then the number $a(\boldsymbol{S},k)$ of
Kekulé structures with exactly $k$ proper sextets is given by 
\begin{equation}
a(\boldsymbol{S},k)=\sum_{\substack{\mathcal{A}\subset\mathcal{S}\\
\left|\mathcal{A}\right|=k
}
}\Omega_{\mathcal{A}}^{\circ}(n),\label{eq:aSk}
\end{equation}
and the ZZ polynomial of $\boldsymbol{S}$ is given by the extended
strict order polynomial of $\mathcal{S}$
\begin{equation}
\text{\emph{ZZ}}(\boldsymbol{S},x)=\text{\emph{E}}_{\mathcal{S}}^{\circ}(n,x+1).\label{eq:ZZSx}
\end{equation}
\end{thm}

\begin{proof}
For every $\mathcal{A}\subset\mathcal{S}$, Lemmata~\ref{thm:Deconstruct_CC}
and \ref{thm:Construct_CC} establish a one-to-one correspondence
between the sets $\mathcal{M}=\left\{ \mu:\mathcal{A}\rightarrow\left[\,n\,\right]~\vert~\mu\text{ is strictly order-preserving}\right\} $
and \linebreak{}
$\mathcal{K}=\left\{ K~\vert~K\text{ is a Kekulé structure with \ensuremath{\mathcal{A}_{K}}=\ensuremath{\mathcal{A}}}\right\} $,
implying that $\left|\mathcal{M}\right|=\left|\mathcal{K}\right|$.
The cardinality of $\mathcal{M}$ is, by Eq.~(\ref{eq:StrictOrderPoly}),
given by $\Omega_{\mathcal{A}}^{\circ}(n)$. In other words, there
are $\Omega_{\mathcal{A}}^{\circ}(n)$ Kekulé structures $K$ with
$\mathcal{A}_{K}=\mathcal{A}$, which directly proves Eq.~(\ref{eq:aSk}).
Then, by Eqs.~(\ref{eq:ZZz}) and (\ref{eq:GenOrderPoly}), we have
\begin{eqnarray}
\text{ZZ}(\boldsymbol{S},x) & = & \sum_{k=0}^{Cl}a(\boldsymbol{S},k)(x+1)^{k}\label{eq:ZZz-1}\\
 & = & \sum_{k=0}^{Cl}\sum_{\substack{\mathcal{A}\subset\mathcal{S}\\
\text{\ensuremath{\left|\mathcal{A}\right|}}=k
}
}\Omega_{\mathcal{A}}^{\circ}(n)(x+1)^{k}\nonumber \\
 & = & \sum_{\substack{\mathcal{A}\subset\mathcal{S}}
}\Omega_{\mathcal{A}}^{\circ}(n)(x+1)^{\text{\ensuremath{\left|\mathcal{A}\right|}}}\nonumber \\
 & = & \text{E}_{\mathcal{S}}^{\circ}(n,x+1).\nonumber 
\end{eqnarray}
\end{proof}

\section{Applications}

Practical application of Theorem\ \ref{thm:PolyEqui} to the determination
of ZZ polynomials of regular strips is presented in Parts 2 and 3
in this series of papers \cite{langner2021guide13,langner2021tables14},
where we give a practical guide to computation of the extended strict
order polynomials $\text{E}_{\mathcal{S}}^{\circ}(n,1+x)$ together
with a complete account of ZZ polynomials $\text{ZZ}(\boldsymbol{S},x)$
of regular $m$-tier benzenoid strips $\boldsymbol{S}$ with $m=1\text{--}6$
and an arbitrary value of $n$ determined as the extended strict order
polynomials $\text{E}_{\mathcal{S}}^{\circ}(n,1+x)$ of the corresponding
posets $\mathcal{S}$. It would be inconvenient to present this collection
of results here owing to its somewhat bulky volume. However, in order
to foreshadow the forthcoming results, we illustrate very briefly
the process of determination of $\text{E}_{\mathcal{S}}^{\circ}(n,1+x)$
using Eq.~(\ref{eq:Znz}) for the two families of benzenoids, $M(2,n)$
and $O(3,2,n)$, for which we constructed the corresponding posets
$\mathcal{S}$ in Examples~\ref{exa:M32poset} and~\ref{exa:O332poset}.
(Note that the poset $\mathcal{S}$ is independent of the structural
parameter $n$, allowing us to compute $\text{E}_{\mathcal{S}}^{\circ}(n,1+x)$
for the whole families of structures simultaneously.) 
\begin{itemize}
\item For $M(2,n)$, the poset $\mathcal{S}$ with $p=2$ vertices in Fig.~\ref{fig2KofM32}
allows only a single linear extension, $\mathcal{L}(\mathcal{S})$=$\left\{ 12\right\} $.
The number of descents and the number of fixed elements in this extension
are both zero, $\text{des}(12)=0$ and $\text{fix}_{\mathcal{S}}(12)=0$.
Consequently, the ZZ polynomial of $M(2,n)$ is given by
\begin{eqnarray*}
\text{ZZ}(M(2,n),x)=\text{E}_{\mathcal{S}}^{\circ}(n,1+x) & = & \sum_{k=0}^{p}\binom{p-\text{fix}_{\mathcal{S}}(12)}{k-\text{fix}_{\mathcal{S}}(12)}\binom{n+\text{des}(12)}{k}\left(1+x\right)^{k}\\
 & = & \sum_{k=0}^{2}\binom{2-0}{k-0}\binom{n+0}{k}\left(1+x\right)^{k}\\
 & = & \sum_{k=0}^{2}\binom{2}{k}\binom{n}{k}\left(1+x\right)^{k}
\end{eqnarray*}
in agreement with, for example, Eq.~(4) of \cite{chou2014closedtextendashform}.
\item For $O(3,2,n)$, the poset $\mathcal{S}$ with $p=6$ vertices in
Fig.~\ref{fig2KofO332} allows five linear extensions, $\mathcal{L}(\mathcal{S})$=$\{123456,$
$123546,132456,135246,132546$\}. The numbers of descents for these
extensions are $0$, $1$, $1$, $1$, and $2$, respectively, and
the numbers of fixed elements are $0$, $2$, $2$, $2$, and $4$,
respectively. Consequently, the ZZ polynomial of $O(3,2,n)$ is given
by
\begin{eqnarray*}
\text{ZZ}(O(3,2,n),x)=\text{E}_{\mathcal{S}}^{\circ}(n,1+x) & \negmedspace\negmedspace=\negmedspace\negmedspace & \sum\limits _{k=0}^{6}\left({\textstyle \binom{6}{k}}{\textstyle \binom{n}{k}}+{\scriptstyle 3}{\textstyle \binom{4}{k-2}}{\textstyle \binom{n+1}{k}}+{\textstyle \binom{2}{k-4}}{\textstyle \binom{n+2}{k}}\right)\left(1+x\right)^{k}
\end{eqnarray*}
in agreement with, for example, Eq.~(25b) of \cite{witek2015zhangzhang}.
\end{itemize}
\noindent Further examples with more details were given previously
in Examples~4 and~5 and Section~5.1 of \cite{langner2020sheep8}. 

\section{Conclusions}

We have demonstrated that for any Kekuléan regular $m$-tier strip
$\boldsymbol{S}$ of length $n$, its Zhang-Zhang polynomial $\text{ZZ}(\boldsymbol{S},x)$
can be computed as the extended strict order polynomial $\text{E}_{\mathcal{S}}^{\circ}(n,x+1)$
\cite{langner2020sheep8} from the poset $\mathcal{S}$ associated
with $\boldsymbol{S}$. The equivalence between $\text{ZZ}(\boldsymbol{S},x)$
and $\text{E}_{\mathcal{S}}^{\circ}(n,x+1)$ given by Theorem~\ref{thm:PolyEqui}
exists owing to the one-to-one correspondence between the set $\left\{ K\right\} $
of Kekulé structures of $\boldsymbol{S}$ and the set $\left\{ \mu:\mathcal{S}\supset\mathcal{A}\rightarrow\left[\,n\,\right]\right\} $
of strictly order-preserving maps from the induced subposets of $\mathcal{S}$
to the interval $\left[\thinspace n\thinspace\right]$ established
by the complementary Lemmata~\ref{thm:Deconstruct_CC} and~\ref{thm:Construct_CC}.
The determination of the poset $\mathcal{S}$ is straightforward and
can be performed directly from the geometrical parameters of $\boldsymbol{S}$
for any Kekuléan regular strip; for non-Kekuléan strips, there is
no poset associated with them, and consequently their ZZ polynomials
are identically equal to 0. Owing to the fact that $\text{E}_{\mathcal{S}}^{\circ}(n,x+1)$
of a $p$-element poset $\mathcal{S}$ can be written in a compact
form as
\begin{equation}
\text{E}_{\mathcal{S}}^{\circ}(n,x+1)=\sum_{w\in\mathcal{L}(\mathcal{S})}\sum_{k=0}^{p}\binom{p-\text{fix}_{\mathcal{S}}(w)}{k-\text{fix}_{\mathcal{S}}(w)}\binom{n+\text{des}(w)}{k}\left(1+x\right)^{k},\label{eq:Esformula}
\end{equation}
the process of determination of $\text{ZZ}(\boldsymbol{S},x)$ can
be completely automatized for any regular $m$-tier strip $\boldsymbol{S}$.
The corresponding algorithm, whose details are elaborated in Part
2 of the current series of papers \cite{langner2021guide13} (see
also \cite{langner2020sheep8} for mathematical details), can be summarized
by the following steps: $\left(i\right)$~Construct the poset $\mathcal{S}$
corresponding to $\boldsymbol{S}$. $\left(ii\right)$~Construct
the set $\mathcal{L}(\mathcal{S})$ of linear extensions of $\mathcal{S}$.
$\left(iii\right)$~For each linear extension $w\in\mathcal{L}(\mathcal{S})$,
compute $\text{des}(w)$ and $\text{fix}_{\mathcal{S}}(w)$. $\left(iv\right)$
Compute the sum in Eq.~(\ref{eq:Esformula}). An associated algorithm
for generating the complete set of Clar covers of $\boldsymbol{S}$
could proceed as follows:~$\left(1\right)$~Construct the poset
$\mathcal{S}$ corresponding to $\boldsymbol{S}$. $\left(2\right)$~Construct
all induced subposets $\mathcal{A}\subset\mathcal{S}$. $\left(3\right)$~For
each induced subposet $\mathcal{A}$, construct a set $\left\{ v\right\} $
of its linear extensions. $\left(4\right)$~For each linear extension
$v$ of $\mathcal{A}$, construct $2^{\left|\mathcal{A}\right|}\thinspace\binom{n+\text{des}(v)}{\left|\mathcal{A}\right|}$
Clar covers by selecting numbers $1\leq k_{1}<k_{2}<\ldots<k_{\left|\mathcal{A}\right|}\leq n+\text{des}(v)$
and assigning to each of the positions $k_{i}$ the covering character
\kh~or \ah. The complete account of ZZ polynomials of regular $m$-tier
benzenoid strips $\boldsymbol{S}$ with $m=1\text{--}6$ and an arbitrary
value of $n$ using the corresponding posets $\mathcal{S}$ and Eq.~(\ref{eq:Esformula})
is presented in Part 3 of the current series of papers \cite{langner2021tables14}.

Summarizing the development presented in the current paper, we want
to stress that the path pursued by us here is unprecedented in the
existing literature on chemical graph theory. We are aware that the
presented results rely heavily on quite advanced concepts in poset
theory and might be difficult to be fully grasped and appreciated
in the first reading. However, in our personal opinion the quite revolutionary
character of our findings deserves particular attention of the community
and should not be overlooked.

\section*{Appendix}

\renewcommand{\thesection}{A}
\setcounter{subsection}{0}

\subsection{\label{subsec:Kek2map}Proof of Lemma~\ref{thm:Deconstruct_CC}}

During this subsection, let $\boldsymbol{S}$ be a regular $m$-tier
strip of length $n$ with the poset of \soc s $\mathcal{S}$ and
a Kekulé structure $K$. In order to prove Lemma~\ref{thm:Deconstruct_CC},
we have to demonstrate that there is exactly one strictly order-preserving
map $\mu:\mathcal{A}_{K}\rightarrow\left[\,n\,\right]$ that satisfies
the two conditions
\begin{enumerate}
\item $\text{pos}_{K}\left(s_{k,j}\right)=\mu\left(s_{k,j}\right)+j$ for
every $s_{k,j}\in\mathcal{A}_{K}$ and
\item $\text{pos}_{K}\left(s_{k,j}\right)=\text{max}\left(\left\{ \mu\left(s_{k',j'}\right)\,\vert\,s_{k',j'}\in\mathcal{A}_{K},s_{k',j'}<_{\mathcal{S}}s_{k,j}\right\} \cup\left\{ 0\right\} \right)+j\text{ for every }s_{k,j}\in\mathcal{S}\setminus\mathcal{A}_{K}.$
\end{enumerate}
\noindent We will show this by explicitly constructing such a map\textemdash first
on the domain $\mathcal{S}$ and later restricted to the domain $\mathcal{A}_{K}$\textemdash and
then by demonstrating its uniqueness.
\begin{defn}
Let $\mu_{K}$ be the map which assigns to each element of $\mathcal{S}$
a number depending on its position in $K$ as follows.
\[
\mu_{K}:\mathcal{S}\rightarrow\mathbb{N},s_{k,j}\mapsto\text{pos}_{K}\left(s_{k,j}\right)-j.
\]
For a given Kekulé structure $K$, we will during this subsection
often denote the number $\mu_{K}(s_{k,j})$ by $m_{k,j}^{K}$, or,
since the considered Kekulé structure is clear, simply by $m_{k,j}$.
\end{defn}

The following lemmata will show that $\mu_{K}$ is an order-preserving
map $\mathcal{S}\rightarrow[\,0,n\,]$, and that the restriction $\mu_{K}\vert_{\mathcal{A}_{K}}$
is a strictly order-preserving map $\mathcal{A}_{K}\rightarrow[\,n\,]$
which is unique for every Kekulé structure.
\begin{lem}
\label{lem:bondsbetween}Consider two \soc s~$s_{k,j},s_{k',j'}\in\mathcal{S}$
in a fragment $f$ with $s_{k',j'}\lessdot_{\mathcal{S}}s_{k,j}$.
Then, $\tau_{K}(s_{k',j'})$ is located to the left of $\tau_{K}(s_{k,j})$,
and the number of interface bonds in $f$ that are located between
$\tau_{K}(s_{k',j'})$ and $\tau_{K}(s_{k,j})$ is given by
\begin{equation}
b_{K}(s_{k',j'},s_{k,j})=2\left(m_{k,j}-m_{k'j'}\right).\label{eq:nrbetween}
\end{equation}
\end{lem}

\begin{proof}
From $s_{k',j'}\lessdot_{\mathcal{S}}s_{k,j}$ it follows according
to Lemma~\ref{<SK} that $\tau_{K}(s_{k',j'})\lessdot_{K}\tau_{K}(s_{k,j})$,
and it is then clear from Theorem~\ref{thm:zigzagsoc} that $\tau_{K}(s_{k',j'})$
must be located to the left of $\tau_{K}(s_{k,j})$. Denote by $i_{l}$
the interface of $f$ which contains the first interface bond of $f$,
and by $i_{r}$ the other interface of $f$. Let further $p=\text{pos}_{K}\left(s_{k,j}\right)=m_{k,j}+j$
and $p'=\text{pos}_{K}\left(s_{k',j'}\right)=m_{k',j'}+j'$, meaning
that $\tau_{K}(s_{k,j})=e_{k,p}$ and $\tau_{K}(s_{k',j'})=e_{k',p'}$.
Following the definition of the bond names, the relevant interface
bonds in $f$ are $\ldots,e_{l,p'},e_{r,p'},e_{l,p'+1},e_{r,p'+1},\ldots,e_{l,p-1},e_{r,p-1},e_{l,p},e_{r,p},\ldots$.
There are two possibilities:
\begin{lyxlist}{00.00.0000}
\item [{$k=r$:}] According to Def.~\ref{def:PosetS}, this implies $j'=j$
and $k'=l$. Thus, the number of interface bonds between $e_{k'p'}\equiv e_{l,p'}$
and $e_{k,p}\equiv e_{r,p}$ is $2(p-p')=2(p-p'-j+j')=2\left(m_{k,j}-m_{k',j'}\right)$.
\item [{$k=l$:}] According to Def.~\ref{def:PosetS}, this implies $j'=j-1$
and $k'=r$. Thus, the number of interface bonds between $e_{k',p'}\equiv e_{r,p'}$
and $e_{k,p}\equiv e_{l,p}$ is $2(p-p'-1)=2(p-p'-j+j')=2\left(m_{k,j}-m_{k',j'}\right)$.
\end{lyxlist}
\noindent In either case, Eq.~(\ref{eq:nrbetween}) is true.
\end{proof}
\begin{lem}
\label{lem:orderpreserving}The map $\mu_{K}$ is order-preserving.
\end{lem}

\begin{proof}
\noindent Consider two \soc s $s_{k,j}$ and $s_{k',j'}$ with $s_{k',j'}<_{\mathcal{S}}s_{k,j}$.
We have to show that the numbers $m_{k,j}$ and $m_{k',j'}$ satisfy
the condition
\begin{equation}
m_{k',j'}\leq m_{k,j}.\label{eq:nkj_leq}
\end{equation}
Assume first that $s_{k',j'}\lessdot_{\mathcal{S}}s_{k,j}$. Then,
according to Lemma~\ref{<SK} we have $\tau_{K}(s_{k',j'})\lessdot_{K}\tau_{K}(s_{k,j})$,
meaning that $\tau_{K}(s_{k',j'})$ and $\tau_{K}(s_{k,j})$ must
belong to the same fragment $f$. The number $b_{K}(s_{k',j'},s_{k,j})$
of interface bonds in $f$ between $\tau_{K}(s_{k',j'})$ and $\tau_{K}(s_{k,j})$
is, according to Lemma~\ref{lem:bondsbetween}, given by $b_{K}(s_{k',j'},s_{k,j})=2m_{k,j}-2m_{k',j'}$.
Since $b_{K}(s_{k',j'},s_{k,j})$ is non-negative, it is clear that
$m_{k',j'}\leq m_{k,j}$.

Assume now that $s_{k',j'}\nlessdot_{\mathcal{S}}s_{k,j}$. Then,
there must exist \soc s $s_{k_{q},j_{q}}$ such that $s_{k',j'}\equiv s_{k_{1},j_{1}}\lessdot_{\mathcal{S}}s_{k_{2},j_{2}}\lessdot_{\mathcal{S}}\ldots\lessdot_{\mathcal{S}}s_{k_{r},j_{r}}\equiv s_{k,j}$.
Since Eq.~(\ref{eq:nkj_leq}) is true for all cover relations, it
follows that $m_{k',j'}\leq m_{k_{2},j_{2}}\leq\ldots\le m_{k_{r-1},j_{r-1}}\le m_{k,j}$.
\end{proof}
\begin{lem}
\label{lem:codomain}The map $\mu_{K}$ has the codomain $[\,0,n\,]$.
\end{lem}

\begin{proof}
Consider a \soc~$s_{k,j}\in\mathcal{S}$. We have to show that its
image $\mu_{K}(s_{k,j})$ satisfies $0\leq\mu_{K}(s_{k,j})\leq n.$
Assume $\tau_{K}(s_{k,j})=e_{k,p}$. In $i_{k}$, to the left of $\tau_{K}(s_{k,j})$,
there are $j-1$ distinct double interface bonds $\tau_{K}(s_{k,1}),\ldots,\tau_{K}(s_{k,j-1})$.
In other words, in $i_{k}$ there are at least $j-1$ bonds to the
left of $e_{k,p}$, meaning $p\geq j$ and thus $\mu_{K}(s_{k,j})=p-j\geq0$
. To the right of $\tau_{K}(s_{k,j})$, in $i_{k}$, there are $\text{ord}(i_{k})-j$
double interface bonds $\tau_{K}(s_{k,j+1}),\ldots,\tau_{K}(s_{k,\text{ord}(i_{k})})$.
With Theorem~\ref{thm:1st rule-1}, it follows that $p\leq\left|i_{k}\right|-(\text{ord}(i_{k})-j)=n+j$
and thus $\mu_{K}(s_{k,j})=p-j\leq n$.
\end{proof}
\begin{lem}
\label{lem:nonproperEq}Let $s_{k,j}\in\mathcal{S}$. The following
statements are equivalent:
\begin{lyxlist}{00.00.0000}
\item [{\emph{$(i)$}}] $m_{k,j}=\text{max}\left(\left\{ m_{k',j'}\,\vert\,s_{k',j'}\in\mathcal{A}_{K},s_{k',j'}<_{\mathcal{S}}s_{k,j}\right\} \cup\left\{ 0\right\} \right)$
\item [{\emph{$(i')$}}] $m_{k,j}=\text{max}\left(\left\{ m_{k',j'}\,\vert\,s_{k',j'}\in\mathcal{S},s_{k',j'}<_{\mathcal{S}}s_{k,j}\right\} \cup\left\{ 0\right\} \right)$
\item [{\emph{$(ii)$}}] $s_{k,j}\in\mathcal{S}\setminus\mathcal{A}_{K}$.
\end{lyxlist}
\end{lem}

\begin{proof}
Assume first that $\tau_{K}(s_{k,j})=e_{k,1}$, which implies $m_{k,j}+j=\text{pos}_{K}(s_{k,j})=1$.
There is no space to the left of $e_{k,1}$ that could accommodate
a proper sextet which has $\tau_{K}(s_{k,j})$ as its double interface
bond. Therefore, statement $(ii)$ is true. Simultaneously, due to
the naming convention of the \soc s, we have $j=1$ and thus $m_{k,j}=1-j=0$.
Since $\mu_{K}$ is order-preserving and maps to $[\,0,n\,]$, it
is clear that $0\leq\text{max}\left\{ m_{k',j'}\,\vert\,s_{k',j'}\in\mathcal{S},s_{k',j'}<_{\mathcal{S}}s_{k,j}\right\} \leq m_{k,j}=0$,
which ensures that the right-hand sides of statements~$(i')$ and
$(i)$ are also equal to zero. Therefore, statements~$(i)$, $(i')$
and $(ii)$ are simultaneously true.

Assume now that $\tau_{K}(s_{k,j})=e_{k,p}$ with $p>1$. This implies
that there is a hexagon~\textcolor{blue}{$H$} of $\boldsymbol{S}$
to the left of $\tau_{K}(s_{k,j})$, as shown on the left side of
Fig.~\ref{fig:Hexa_all_coverings}. It is easy to verify that in
any Kekulé structure, there are only five possible coverings of the
bonds in and around \textcolor{blue}{$H$}, which are depicted on
the right side of Fig.~\ref{fig:Hexa_all_coverings}.
\begin{figure}[H]
\noindent \begin{centering}
\includegraphics[scale=0.32]{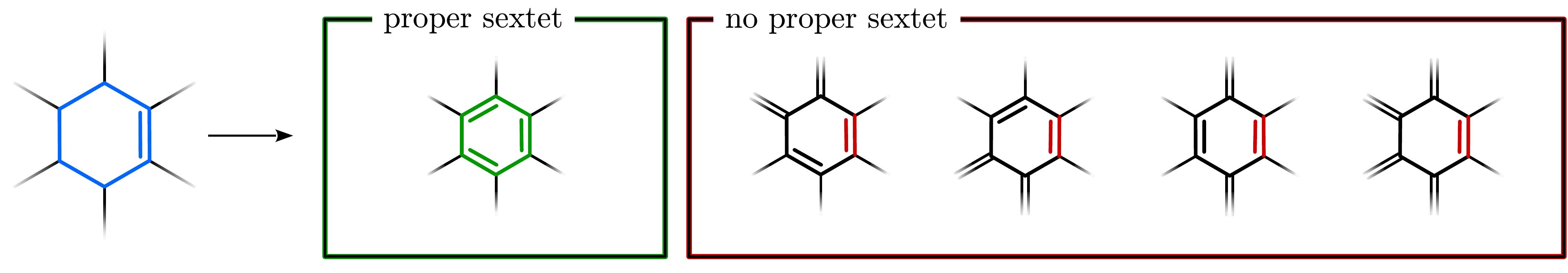}
\par\end{centering}
\caption{\label{fig:Hexa_all_coverings}All possible coverings of the bonds
in and around a hexagon \textcolor{blue}{$H$} directly to the left
of a double bond $s_{k,j}$.}
\end{figure}
Only in one of the five possible cases we find that the covering of
\textcolor{blue}{$H$} forms a proper sextet. This is the case exactly
if none of the interface bonds connected to the top and bottom corner
of \textcolor{blue}{$H$} are double bonds. In the other four cases,
at least one these bonds is a double bond, and $\tau_{K}(s_{k,j})$
is not part of a proper sextet, meaning that $s_{k,j}\in\mathcal{S}\setminus\mathcal{A}_{K}$.
As a result, the following statements are equivalent.
\begin{lyxlist}{00.00.0000}
\item [{$(ii)$}] $s_{k,j}\in\mathcal{S}\setminus\mathcal{A}_{K}$.
\item [{\,~\raisebox{-0.35\height}{\rotatebox{90}{$\Longleftrightarrow$}}}] (according
to Fig.~\ref{fig:Hexa_all_coverings})
\item [{$(b)$}] There is a double bond connected to the top or bottom
corner of \textcolor{blue}{$H$}.
\item [{~\,\raisebox{-0.35\height}{\rotatebox{90}{$\Longleftrightarrow$}}}]~
\item [{$(c)$}] At least one of the following is true:
\begin{lyxlist}{00.00.0000}
\item [{$(c')$}] In the fragment $f_{k}$, there is a \soc~$s_{k-1,j'}\lessdot_{\mathcal{S}}s_{k,j}$
such that there are no interface bonds of $f_{k}$ between $\tau_{K}(s_{k-1,j'})$
and $\tau_{K}(s_{k,j})$.
\item [{$(c'')$}] In the fragment $f_{k+1}$, there is a \soc~$s_{k+1,j''}\lessdot_{\mathcal{S}}s_{k,j}$
such that there are no interface bonds of $f_{k+1}$ between $\tau_{K}(s_{k+1,j''})$
and $\tau_{K}(s_{k,j})$.
\end{lyxlist}
\item [{\,~\raisebox{-0.35\height}{\rotatebox{90}{$\Longleftrightarrow$}}}] (by
Lemma~\ref{lem:bondsbetween})
\item [{$(d)$}] At least one of the following is true:
\begin{lyxlist}{00.00.0000}
\item [{$(d')$}] In the fragment $f_{k}$, there is a \soc~$s_{k-1,j'}\lessdot_{\mathcal{S}}s_{k,j}$
such that $m_{k-1,j'}=m_{k,j}$.
\item [{$(d'')$}] In the fragment $f_{k+1}$, there is a \soc~$s_{k+1,j''}\lessdot_{\mathcal{S}}s_{k,j}$
such that $m_{k+1,j''}=m_{k,j}$.
\end{lyxlist}
\item [{~\,\raisebox{-0.35\height}{\rotatebox{90}{$\Longleftrightarrow$}}}] (since
$0\leq m_{k'j'}\leq m_{kj}$ for all $s_{k',j'}\in\mathcal{S}$ with
$s_{k',j'}<_{\mathcal{S}}s_{k,j}$)
\item [{$(i')$}] $m_{k,j}=\text{max}\left\{ m_{k',j'}\,\vert\,s_{k',j'}\in\mathcal{S},s_{k',j'}<_{\mathcal{S}}s_{k,j}\right\} $
\\
$\phantom{m_{k,j}}=\text{max}\left(\left\{ m_{k',j'}\,\vert\,s_{k',j'}\in\mathcal{S},s_{k',j'}<_{\mathcal{S}}s_{k,j}\right\} \cup\left\{ 0\right\} \right)$.
\end{lyxlist}
It remains to be shown that statements $(i')$ and $(i)$ of the lemma
are equivalent. The implication $(i)\Rightarrow(i')$ is obvious from
the inclusion $\mathcal{\mathcal{A}}_{K}\subset\mathcal{S}$ together
with the fact that $\mu_{K}$ is order-preserving. We still have to
show $(i')\Rightarrow(i)$, i.e., that for the considered above element
$s_{k,j}\in\mathcal{S}\setminus\mathcal{A}_{K}$, we have 
\begin{equation}
m_{k,j}=\text{max}\left(\left\{ m_{k',j'}\,\vert\,s_{k',j'}\in\mathcal{\mathcal{A}}_{K},s_{k',j'}<_{\mathcal{S}}s_{k,j}\right\} \cup\left\{ 0\right\} \right).\label{eq:mkjAk}
\end{equation}
We distinguish two cases: $m_{k,j}=0$ and $m_{k,j}>0$. Keep in mind
that we know from Lemmata~\ref{lem:orderpreserving} and \ref{lem:codomain}
that 
\begin{equation}
m_{k,j}\geq\text{max}\left(\left\{ m_{k',j'}\,\vert\,s_{k',j'}\in\mathcal{\mathcal{A}}_{K},s_{k',j'}<_{\mathcal{S}}s_{k,j}\right\} \cup\left\{ 0\right\} \right).\label{eq:mkjgeq}
\end{equation}

In the case $m_{k,j}=0$, it follows from Eq.~(\ref{eq:mkjgeq})
that Eq.~(\ref{eq:mkjAk}) is trivially true.

In the case $m_{k,j}>0$, consider the set $M_{k,j}=\left\{ s_{k',j'}\in\mathcal{S}\,\vert\,m_{k',j'}=m_{k,j},s_{k',j'}<_{\mathcal{S}}s_{k,j}\right\} $.
Statement $(i')$ together with $m_{k,j}>0$ implies that $M_{k,j}$
is non-empty, and thus contains at least one minimal (w.r.t. $<_{\mathcal{S}}$)
element $s_{\tilde{k},\tilde{j}}$. We will show that $s_{\tilde{k},\tilde{j}}\in\mathcal{A}_{K}$.
Clearly, we have $m_{\tilde{k},\tilde{j}}=m_{k,j}>0$. Furthermore,
by choice of $s_{\tilde{k},\tilde{j}}$ as a minimal element of $M_{k,j}$,
for any $s_{k',j'}\in\mathcal{S}$ with $s_{k',j'}<_{\mathcal{S}}s_{\tilde{k},\tilde{j}}$,
we must have $m_{k',j'}<m_{\tilde{k},\tilde{j}}$. It follows that
$m_{\tilde{k},\tilde{j}}>\text{max}\left(\left\{ m_{k',j'}\,\vert\,s_{k',j'}\in\mathcal{S},s_{k',j'}<_{\mathcal{S}}s_{\tilde{k},\tilde{j}}\right\} \cup\left\{ 0\right\} \right)$,
which means, due to the equivalence of the statements~$(ii)$ and
$(i')$, that $s_{\tilde{k},\tilde{j}}\in\mathcal{A}_{K}$. We have
shown that there exists an element $s_{\tilde{k},\tilde{j}}\in\mathcal{A}_{K}$
with $s_{\tilde{k},\tilde{j}}<_{\mathcal{S}}s_{k,j}$ and $m_{\tilde{k},\tilde{j}}=m_{k,j}$;
thus it follows together with Eq.~(\ref{eq:mkjgeq}) that $m_{k,j}=\text{max}\left\{ m_{k',j'}\,\vert\,s_{k',j'}\in\mathcal{A}_{K},s_{k',j'}<_{\mathcal{S}}s_{k,j}\right\} $,
which shows that Eq.~(\ref{eq:mkjAk}) holds. Therefore, in either
case, it follows from statement $(i')$ that statement $(i)$ is true.
\end{proof}
\begin{lem}
\label{lem:strictn}The restriction of the map $\mu_{K}$ to the subposet
$\mathcal{A}_{K}\subset\mathcal{S}$ is a strictly order-preserving
map $\mu_{K}\vert_{\mathcal{A}_{K}}:\,\mathcal{A}_{K}\rightarrow[\,n\,]$.
\end{lem}

\begin{proof}
Consider an element $s_{k,j}\in\mathcal{A}_{K}$. We know from the
order-preserving nature and the codomain of $\mu_{K}$ that $m_{k,j}\geq\text{max}\left(\left\{ m_{k',j'}\,\vert\,s_{k',j'}\in\mathcal{S},s_{k',j'}<_{\mathcal{S}}s_{k,j}\right\} \cup\left\{ 0\right\} \right)$
and from Lemma~\ref{lem:nonproperEq} that $m_{k,j}\neq\text{max}\left(\left\{ m_{k',j'}\,\vert\,s_{k',j'}\in\mathcal{S},s_{k',j'}<_{\mathcal{S}}s_{k,j}\right\} \cup\left\{ 0\right\} \right)$;
thus $m_{k,j}>\text{max}\left(\left\{ m_{k',j'}\,\vert\,s_{k',j'}\in\mathcal{S},s_{k',j'}<_{\mathcal{S}}s_{k,j}\right\} \cup\left\{ 0\right\} \right)$
is the only remaining possibility. This implies two facts: Firstly,
$m_{k,j}>0$, meaning that the codomain of $\mu_{K}\vert_{\mathcal{A}_{K}}$
does not contain zero and is therefore restricted to $[\,n\,]$. Secondly,
$m_{k,j}>m_{k',j'}$ for every $s_{k',j'}\in\mathcal{S}$ with $s_{k',j'}<_{\mathcal{S}}s_{k,j}$\textemdash and
thus in particular for every such $s_{k',j'}\in\mathcal{A}_{K}$.
Since $\mu_{K}(s_{k',j'})\equiv m_{k',j'}<m_{k,j}\equiv\mu_{K}(s_{k,j})$
for every $s_{k',j'}\in\mathcal{A}_{K}$ with $s_{k',j'}<_{\mathcal{S}}s_{k,j}$,
the restricted map $\mu_{K}\vert_{\mathcal{A}_{K}}$ is strictly order-preserving.
\end{proof}
\begin{lem}
\label{lem:muunique}Let $\mu$ be a strictly order-preserving map
$\mu:\mathcal{A}_{K}\rightarrow\left[\,n\,\right]$ which satisfies
conditions $1$ and $2$ of Lemma~\ref{thm:Deconstruct_CC}. Then,
$\mu=\mu_{K}\vert_{\mathcal{A}_{K}}$.
\end{lem}

\begin{proof}
The maps $\mu$ and $\mu_{K}\vert_{\mathcal{A}_{K}}$ have the same
domain~$\mathcal{A}_{K}$. For every $s_{k,j}\in\mathcal{A}_{K}$,
condition $1$ of Lemma~\ref{thm:Deconstruct_CC} fully determines
the value of $\mu\left(s_{k,j}\right)$ to be $\mu\left(s_{k,j}\right)=\text{pos}_{K}\left(s_{k,j}\right)-j=\mu_{K}\left(s_{k,j}\right)$.
\end{proof}
\begin{proof}
\emph{(of Lemma~\ref{thm:Deconstruct_CC})} The map $\mu_{K}\vert_{\mathcal{A}_{K}}$
is, according to Lemma~\ref{lem:strictn}, a strictly order-preserving
map $\mu_{K}\vert_{\mathcal{A}_{K}}:\,\mathcal{A}_{K}\rightarrow[\,n\,]$.
It satisfies condition~1 of Lemma~\ref{thm:Deconstruct_CC} by construction.
Furthermore, this map satisfies condition~2 of Lemma~\ref{thm:Deconstruct_CC}
according to Lemma~\ref{lem:nonproperEq}. Finally, Lemma~\ref{lem:muunique}
demonstrates that $\mu_{K}\vert_{\mathcal{A}_{K}}$ is the only strictly
order-preserving map $\mathcal{A}_{K}\rightarrow[\,n\,]$ which satisfies
both conditions.
\end{proof}

\subsection{Proof of Lemma~\ref{thm:Construct_CC}}

The purpose of the present subsection is to prove Lemma~\ref{thm:Construct_CC}.
During this subsection, let $\boldsymbol{S}$ be a Kekuléan regular
$m$-tier strip of length $n$ with the set of \soc s $\mathcal{S}$.
Let further $\mathcal{A}\subset\mathcal{S}$ be an induced subposet
of $\mathcal{S}$, and let $\mu:\mathcal{A}\rightarrow\left[\,n\,\right]$
be a strictly order-preserving map. Recall that $\mu_{K}\left(s_{k,j}\right)=\text{pos}_{K}\left(s_{k,j}\right)-j$
for every $s_{k,j}\in\mathcal{S}.$ We have to show that there exists
exactly one Kekulé structure $K$ of $\boldsymbol{S}$ such that
\begin{enumerate}
\item $\mu_{K}\left(s_{k,j}\right)=\mu\left(s_{k,j}\right)$ for every $s_{k,j}\in\mathcal{A}$,
\item $\mu_{K}\left(s_{k,j}\right)=\text{\emph{\ensuremath{\max}}}\left(\left\{ \mu\left(s_{k',j'}\right)\,\vert\,s_{k',j'}\in\mathcal{A},s_{k',j'}<_{\mathcal{S}}s_{k,j}\right\} \cup\left\{ 0\right\} \right)\text{ for every }s_{k,j}\in\mathcal{S}\setminus\mathcal{A}$,
\item $\mathcal{A}_{K}=\mathcal{A}$.
\end{enumerate}
The proof of Lemma~\ref{thm:Construct_CC} proceeds by explicitly
constructing the Kekulé structure in question. The following Lemma~\ref{lem:Uniqueness}
will be used to demonstrate the uniqueness of such a Kekulé structure.
\begin{lem}
\label{lem:Uniqueness}Consider two Kekulé structures $K$ and $K'$
of $\boldsymbol{S}$ with $\mathcal{A}_{K}=\mathcal{A}_{K'}$ and
$\mu_{K}\vert_{\mathcal{A}_{K}}=\mu_{K'}\vert_{\mathcal{A}_{K'}}$.
Then, $K=K'$.
\end{lem}

\begin{proof}
For every $s_{k,j}\in\mathcal{A}_{K}$, we know that $\mu_{K}(s_{k,j})=\mu_{K'}(s_{k,j})$.
For every $s_{k,j}\in\mathcal{S}\setminus\mathcal{A}_{K}$, we know
from the assumptions and Lemma~\ref{lem:nonproperEq} that 
\begin{eqnarray*}
\mu_{K}(s_{k,j}) & = & \text{max}\left(\left\{ \mu_{K\phantom{'}}(s_{k',j'})\,\vert\,s_{k',j'}\in\mathcal{A}_{K\phantom{'}},s_{k',j'}<_{\mathcal{S}}s_{k,j}\right\} \cup\left\{ 0\right\} \right)\\
 & = & \text{max}\left(\left\{ \mu_{K'}(s_{k',j'})\,\vert\,s_{k',j'}\in\mathcal{A}_{K'},s_{k',j'}<_{\mathcal{S}}s_{k,j}\right\} \cup\left\{ 0\right\} \right)\,\,\,\,=\,\,\,\,\mu_{K'}(s_{k,j}).
\end{eqnarray*}
 Therefore, we have $\mu_{K}(s_{k,j})=\mu_{K'}(s_{k,j})$ for every
$s_{k,j}\in\mathcal{S}$, which means that the double interface bonds
in $K$ and $K'$ are identical, and thus
\[
K_{I}=\left\{ e_{k,\mu_{K}\left(s_{k,j}\right)+j}\,\vert\,s_{k,j}\in\mathcal{S}\right\} =\left\{ e_{k,\mu_{K'}\left(s_{k,j}\right)+j}\,\vert\,s_{k,j}\in\mathcal{S}\right\} =K_{I}'.
\]
It follows from Lemma~\ref{lem:Fulldet} that $K=K'$.
\end{proof}
\begin{proof}
\emph{(of Lemma~\ref{thm:Construct_CC})} We will first construct
a set ${K_{v}}\subset\bigcup_{k}i_{k}$ of interface bonds of $\boldsymbol{S}$
that corresponds to $\mathcal{A}$ and $\mu$. Then, we show that
the third rule of interface theory given in Theorem~\ref{thm:Construct_CC-1}
can be applied to prove that there exists a Kekulé structure $K$
with $K_{I}={K_{v}}$.

Let us consider the auxiliary map $\tilde{\mu}:\mathcal{S\rightarrow}[\,0,n\,]$
constructed as follows.
\begin{equation}
\tilde{\mu}(s_{k,j})=\begin{cases}
\mu(s_{k,j}) & \text{if }s_{k,j}\in\mathcal{A}\\
\text{max}\left(\left\{ \mu(s_{k',j'})\,\vert\,s_{k',j'}\in\mathcal{A},s_{k',j'}<_{\mathcal{S}}s_{k,j}\right\} \cup\left\{ 0\right\} \right) & \text{if }s_{k,j}\in\mathcal{\mathcal{S}\setminus\mathcal{A}}
\end{cases}\label{eq:mutilde}
\end{equation}
During this subsection, let us for all $s_{k,j}\in\mathcal{\mathcal{S}}$
denote the value $\tilde{\mu}(s_{k,j})$ by $\tilde{m}_{k,j}$. Clearly,
the map $\tilde{\mu}:\mathcal{S}\rightarrow[\,0,n\,]$, $s_{k,j}\mapsto\tilde{m}_{k,j}$
is order-preserving. For every $s_{k,j}\in\mathcal{S}$, let $n_{k,j}=\tilde{m}_{k,j}+j$.

\emph{The numbers $n_{k,j}$ specify distinct, well-defined interface
bonds $e_{k,n_{k,j}}$:} For any given $k\in[\,m\,]$, the set $\left\{ s_{k',j}\in\mathcal{S}\,\vert\,k'=k\right\} $
contains, by construction, $\text{ord}(i_{k})$ elements $s_{k,1}<_{\mathcal{S}}s_{k,2}<_{\mathcal{S}}\ldots<_{\mathcal{S}}s_{k,\text{ord}(i_{k})}$.
Since $\tilde{\mu}$ is order-preserving and has the codomain $\tilde{\mu}:\mathcal{S}\rightarrow[\,0,n\,]$,
it follows that the corresponding numbers $\tilde{m}_{k,j}$ satisfy
$0\leq\tilde{m}_{k,1}\leq\tilde{m}_{k,2}\leq\ldots\leq\tilde{m}_{k,\text{ord}(i_{k})}\leq n$,
and consequently we have $0<n_{k,1}<n_{k,2}<\ldots<n_{k,\text{ord}(i_{k})}\leq n+\text{ord}(i_{k})$.
The interface $i_{k}$ contains, according to Theorem~\ref{thm:1st rule-1},
$\left|i_{k}\right|=n+\text{ord}(i_{k})$ bonds $e_{k,1},\ldots,e_{k,n+\text{ord}(i_{k})}$.
Therefore, for every $s_{k,j}$, the bond $e_{k,n_{k,j}}$ is an element
of $i_{k}$ and is different from all other $e_{k,n_{k,j'}}$ with
$j\neq j'$. 

Let us now set $K_{v}=\left\{ e_{k,n_{k,j}}\,\vert\,s_{k,j}\in\mathcal{S}\right\} $.
We have ensured by construction that every interface $i_{k}$ contains
exactly $\text{ord}(i_{k})$ elements of $K_{v}$: $\left|K_{v}\cap i_{k}\right|=\text{ord}(i_{k})$,
meaning that \emph{the set $K_{v}$ satisfies condition $(a)$ of
the Third Rule of interface theory given in Theorem~\ref{thm:Construct_CC-1}.}

Next, we need to show that \emph{the set $K_{v}$ satisfies conditions
$(a)-(c)$ of the Second Rule of interface theory given in Theorem~\ref{thm:2nd rule-1}
(and consequently condition $(b)$ of the Third Rule):} Consider an
arbitrary fragment $f$ of $\boldsymbol{B}$, denote by $i_{l}$ the
interface of $f$ which contains the first interface bond of $f$,
and denote by $i_{r}$ the other interface of $f$. According to Def.~\ref{def:PosetS},
the \soc s of $f$ satisfy, for applicable values of $j=1,2,\ldots$,
the cover relations $s_{l,j}\lessdot_{\mathcal{S}}s_{r,j}$ and $s_{r,j}\lessdot_{\mathcal{S}}s_{l,j+1}$.
Since $\tilde{\mu}$ is order-preserving, it follows that $\tilde{m}_{l,j}\leq\tilde{m}_{r,j}$
and $\tilde{m}_{r,j}\leq\tilde{m}_{l,j+1}$, and therefore $n_{l,j}\leq n_{r,j}$
and $n_{r,j}<n_{l,j+1}$. Note now that the interface bond $e_{l,p}$
is located to the left of the interface bond $e_{r,p}$, which in
turn is located to the left of $e_{l,p+1}$; therefore, $n_{l,j}\leq n_{r,j}$
ensures that $e_{l,n_{l,j}}$ is located to the left of $e_{r,n_{r,j}}$,
and $n_{r,j}<n_{l,j+1}$ ensures that $e_{r,n_{r,j}}$ is located
to the left of $e_{l,n_{l,j+1}}$. Thus, the sequence of double interface
bonds of the fragment $f$, from left to right, is given by
\[
e_{l,n_{l,1}},e_{r,n_{r,1}},e_{l,n_{l,2}},e_{r,n_{r,2}},\ldots,\begin{cases}
e_{r,n_{r,\text{ord}(i_{r})}} & \text{if }\text{ord}(i_{l})+\text{ord}(i_{r})\text{ is even},\\
e_{l,n_{l,\text{ord}(i_{l})}} & \text{if }\text{ord}(i_{l})+\text{ord}(i_{r})\text{ is odd}.
\end{cases}
\]
Comparison to the sequence of interface bonds of $f$
\[
e_{l,1},e_{r,1},e_{l,2},e_{r,2},\ldots,\begin{cases}
e_{r,\text{ord}(i_{r})} & \text{if }\text{ord}(i_{l})+\text{ord}(i_{r})\text{ is even,}\\
e_{l,\text{ord}(i_{l})} & \text{if }\text{ord}(i_{l})+\text{ord}(i_{r})\text{ is odd,}
\end{cases}
\]
shows that all the conditions of the Second Rule, Theorem~\ref{thm:2nd rule-1}
are satisfied. 

With this, we have seen that the two conditions $(a)$ and $(b)$
of the Third Rule of interface theory given in Theorem~\ref{thm:Construct_CC-1}
are satisfied; consequently, Theorem~\ref{thm:Construct_CC-1} states
there exists exactly one Kekulé structure $K$ with $K_{I}=K_{v}$.

It remains to show that this Kekulé structure $K$ indeed satisfies
conditions 1\textendash 3 of Lemma~\emph{\ref{thm:Construct_CC}}.
It is clear from the construction of $K_{v}$ and the fact that $K_{I}=K_{v}$
that $\tilde{\mu}=\mu_{K}$ (because $\mu_{K}(s_{k,j})=\text{pos}_{K}\left(s_{k,j}\right)-j=n_{k,j}-j=\tilde{m}_{k,j}=\tilde{\mu}(s_{k,j})$
for all $s_{k,j}\in\mathcal{S}$ and the domains of both maps are
identical). First, we show condition 3 by demonstrating that $\mathcal{A}=\mathcal{A}_{K}$:

By construction of $\mu$ and $\tilde{\mu}$, for all $s_{k,j}\in\mathcal{A}$,
we have $\tilde{\mu}(s_{k,j})=\mu(s_{k,j})\in\left[\thinspace n\thinspace\right]$
and
\begin{equation}
\tilde{\mu}(s_{k,j})>\text{max}\left(\left\{ \tilde{\mu}(s_{k',j'})\,\vert\,s_{k',j'}\in\mathcal{S},s_{k',j'}<_{\mathcal{S}}s_{k,j}\right\} \cup\left\{ 0\right\} \right)\label{eq:mutildeL}
\end{equation}

\noindent for the following reason: For every $s_{k',j'}\in\mathcal{A}$,
we have $\tilde{\mu}(s_{k',j'})=\mu(s_{k',j'})<\mu(s_{k,j})=\tilde{\mu}(s_{k,j})$
since $\mu$ is strictly order-preserving. For every $s_{k',j'}\in\mathcal{S}\setminus\mathcal{A}$,
by Eq.~(\ref{eq:mutilde}) there are two possibilities: $\tilde{\mu}(s_{k',j'})=0<\mu(s_{k,j})=\tilde{\mu}(s_{k,j})$,
or there exists an element $s_{k'',j''}\in\mathcal{A}$ with $s_{k'',j''}<_{\mathcal{S}}s_{k',j'}$
for which $\tilde{\mu}(s_{k',j'})=\mu(s_{k'',j''})$. Since $s_{k'',j''}\in\mathcal{A}$
and since $\mu$ is strictly order-preserving, we have $\tilde{\mu}(s_{k',j'})=\mu(s_{k'',j''})<\mu(s_{k,j})=\tilde{\mu}(s_{k,j})$.
It follows from Eq.~(\ref{eq:mutildeL}) and $\tilde{\mu}=\mu_{K}$
that statement $(i')$ of Lemma~\ref{lem:nonproperEq} is not satisfied
for $s_{k,j}$, which shows that $s_{k,j}\in\mathcal{A}_{K}$, demonstrating
that $\mathcal{A}\subset\mathcal{A}_{K}$.

On the other hand, for every $s_{k,j}\in\mathcal{S}\setminus\mathcal{A}$,
we have by construction in Eq.~(\ref{eq:mutildeL})
\begin{equation}
\tilde{\mu}(s_{k,j})=\text{max}\left(\left\{ \mu(s_{k',j'})\,\vert\,s_{k',j'}\in\mathcal{A},s_{k',j'}<_{\mathcal{S}}s_{k,j}\right\} \cup\left\{ 0\right\} \right).\label{eq:mutildeE}
\end{equation}
Since $\tilde{\mu}$ is order-preserving, any $s_{k',j'}\in\mathcal{S}$
with $s_{k',j'}<_{\mathcal{S}}s_{k,j}$ must satisfy $\tilde{\mu}(s_{k',j'})\leq\tilde{\mu}(s_{k,j})$,
and therefore it is clear that 
\begin{equation}
\tilde{\mu}(s_{k,j})=\text{max}\left(\left\{ \tilde{\mu}(s_{k',j'})\,\vert\,s_{k',j'}\in\mathcal{S},s_{k',j'}<_{\mathcal{S}}s_{k,j}\right\} \cup\left\{ 0\right\} \right).\label{eq:mutildeES}
\end{equation}
It follows from Eq.~(\ref{eq:mutildeES}) and $\tilde{\mu}=\mu_{K}$
that statement $(i')$ of Lemma~\ref{lem:nonproperEq} is satisfied
for $s_{k,j}$, which shows that $s_{k,j}\in\mathcal{S}\setminus\mathcal{A}_{K}$,
demonstrating that $\mathcal{S}\setminus\mathcal{A}\subset\mathcal{S}\setminus\mathcal{A}_{K}$
and consequently $\mathcal{A}_{K}\subset\mathcal{A}$.

This shows that $\mathcal{A}=\mathcal{A}_{K}$, demonstrating correctness
of statement 3 of Lemma~\ref{thm:Construct_CC}.

It is now obvious that
\begin{eqnarray}
\mu_{K}\left(s_{k,j}\right)=\tilde{\mu}(s_{k,j}) & = & \begin{cases}
\mu(s_{k,j}) & \text{if }s_{k,j}\in\mathcal{A}\\
\text{max}\left(\left\{ \mu(s_{k',j'})\,\vert\,s_{k',j'}\in\mathcal{A},s_{k',j'}<_{\mathcal{S}}s_{k,j}\right\} \cup\left\{ 0\right\} \right) & \text{if }s_{k,j}\in\mathcal{\mathcal{S}\setminus\mathcal{A}}
\end{cases}\label{eq:cond12proof}\\
 & = & \begin{cases}
\mu(s_{k,j}) & \text{if }s_{k,j}\in\mathcal{A}_{K}\\
\text{max}\left(\left\{ \mu(s_{k',j'})\,\vert\,s_{k',j'}\in\mathcal{A},s_{k',j'}<_{\mathcal{S}}s_{k,j}\right\} \cup\left\{ 0\right\} \right) & \text{if }s_{k,j}\in\mathcal{\mathcal{S}\setminus\mathcal{A}}_{K}
\end{cases}
\end{eqnarray}
which demonstrates correctness of statements 1 and 2 of Lemma~\ref{thm:Construct_CC}.

Every Kekulé structure $K'$ satisfying conditions $1$\textendash $3$
has to satisfy $\mu_{K'}\vert_{\mathcal{A}}=\mu=\mu_{K}\vert_{\mathcal{A}}$,
and thus according to Lemma~\ref{lem:Uniqueness} is identical to
$K$, ensuring that there is only one such Kekulé structure, concluding
the proof.
\end{proof}
\noindent \textbf{Acknowledgement.} This work was financially supported
by Ministry of Science and Technology of Taiwan (MOST108-2113-M-009-010-MY3)
and the Center for Emergent Functional Matter Science of National
Chiao Tung University from the Featured Areas Research Center Program
within the framework of the Higher Education Sprout Project by the
Ministry of Education (MOE), Taiwan.

\end{document}